\pdfoutput=1
\documentclass[a4paper,11pt]{amsart}
\usepackage{amsmath,amssymb,url,stmaryrd}
\usepackage[all]{xy}
\usepackage[pdftex]{graphicx}
\usepackage[pdftex]{color}
\usepackage[top=30truemm,bottom=30truemm,left=25truemm,right=25truemm]{geometry}

\newtheorem{Thm}{Theorem}[section]
\newtheorem{Cor}[Thm]{Corollary}
\newtheorem{Prop}[Thm]{Proposition}
\newtheorem{Lem}[Thm]{Lemma}
\newtheorem*{Thm*}{Theorem}

\theoremstyle{definition}
\newtheorem{Def}[Thm]{Definition}

\newtheorem{Ex}[Thm]{Example}

\theoremstyle{remark}
\newtheorem{Rem}[Thm]{Remark}

\makeatletter

\@addtoreset{equation}{section}
\makeatother

\DeclareMathOperator{\gcdsf}{\mathsf{gcd}}
\DeclareMathOperator{\module}{\mathsf{mod}}

\DeclareMathOperator{\silt}{\mathsf{silt}}
\DeclareMathOperator{\twoipresilt}{\mathsf{2-ipresilt}}
\DeclareMathOperator{\twopresilt}{\mathsf{2-presilt}}
\DeclareMathOperator{\twosilt}{\mathsf{2-silt}}
\DeclareMathOperator{\smc}{\mathsf{smc}}
\DeclareMathOperator{\twosmc}{\mathsf{2-smc}}
\DeclareMathOperator{\tstr}{\mathsf{t-str}}
\DeclareMathOperator{\inttstr}{\mathsf{int-t-str}}

\DeclareMathOperator{\ftors}{\mathsf{f-tors}}
\DeclareMathOperator{\ftorf}{\mathsf{f-torf}}
\DeclareMathOperator{\tors}{\mathsf{tors}}

\DeclareMathOperator{\wide}{\mathsf{wide}}
\DeclareMathOperator{\Fac}{\mathsf{Fac}}

\DeclareMathOperator{\supp}{\mathsf{supp}}
\DeclareMathOperator{\Sub}{\mathsf{Sub}}

\DeclareMathOperator{\brick}{\mathsf{brick}}
\DeclareMathOperator{\sbrick}{\mathsf{sbrick}}
\DeclareMathOperator{\fLsbrick}{\mathsf{f_L-sbrick}}
\DeclareMathOperator{\fRsbrick}{\mathsf{f_R-sbrick}}

\DeclareMathOperator{\inj}{\mathsf{inj}}
\DeclareMathOperator{\proj}{\mathsf{proj}}

\DeclareMathOperator{\Hom}{\mathsf{Hom}}

\DeclareMathOperator{\End}{\mathsf{End}}

\DeclareMathOperator{\add}{\mathsf{add}}

\DeclareMathOperator{\Ker}{\mathsf{Ker}}
\DeclareMathOperator{\Image}{\mathsf{Im}}
\DeclareMathOperator{\Coker}{\mathsf{Coker}}

\DeclareMathOperator{\dimension}{\mathsf{dim}}
\DeclareMathOperator{\dimv}{\underline{\mathsf{dim}}}
\DeclareMathOperator{\Mat}{\mathsf{Mat}}

\DeclareMathOperator{\red}{\mathsf{red}}

\newcommand{\Q}{\mathbb{Q}}
\newcommand{\R}{\mathbb{R}}
\newcommand{\Z}{\mathbb{Z}}

\newcommand{\calC}{\mathcal{C}}
\newcommand{\calD}{\mathcal{D}}
\newcommand{\calF}{\mathcal{F}}
\newcommand{\calH}{\mathcal{H}}
\newcommand{\calS}{\mathcal{S}}
\newcommand{\calT}{\mathcal{T}}
\newcommand{\calU}{\mathcal{U}}
\newcommand{\calV}{\mathcal{V}}
\newcommand{\calW}{\mathcal{W}}
\newcommand{\calX}{\mathcal{X}}

\newcommand{\ovcalT}{\overline{\calT}}
\newcommand{\ovcalF}{\overline{\calF}}

\newcommand{\rmb}{\mathrm{b}}

\newcommand{\sfD}{\mathsf{D}}
\newcommand{\sfF}{\mathsf{F}}
\newcommand{\sfK}{\mathsf{K}}

\newcommand{\sfT}{\mathsf{T}}

\newcommand{\Wall}{\mathsf{Wall}}
\newcommand{\Cone}{\mathsf{Cone}}
\newcommand{\Chamber}{\mathsf{Chamber}}
\newcommand{\TF}{\mathsf{TF}}

\newcommand{\vecc}{\boldsymbol{c}}
\newcommand{\vecd}{\boldsymbol{d}}

\renewcommand{\gcd}{\gcdsf}
\renewcommand{\mod}{\module}

\renewcommand{\Gamma}{\varGamma}
\renewcommand{\Delta}{\varDelta}
\renewcommand{\epsilon}{\varepsilon}
\renewcommand{\Theta}{\varTheta}
\renewcommand{\Lambda}{\varLambda}
\renewcommand{\Pi}{\varPi}
\renewcommand{\Sigma}{\varSigma}
\renewcommand{\phi}{\varphi}
\renewcommand{\Omega}{\varOmega}

\renewcommand{\Im}{\Image}
\renewcommand{\dim}{\dimension}

\newcommand{\ang}[2]{\langle #1, #2 \rangle}

\begin{document}

\title{The wall-chamber structures of the real Grothendieck groups}
\author{Sota Asai} 
\address{Department of Pure and Applied Mathematics, Graduate School of Information Science and Technology, Osaka University, 1-5 Yamadaoka, Suita-shi, Osaka-fu, 565-0871, Japan}
\email{s-asai@ist.osaka-u.ac.jp}

\begin{abstract}
For a finite-dimensional algebra $A$ over a field $K$
with $n$ simple modules, 
the real Grothendieck group 
$K_0(\proj A)_\R:=K_0(\proj A) \otimes_\Z \R \cong \R^n$ 
gives stability conditions of King.
We study the associated wall-chamber structure of 
$K_0(\proj A)_\R$ by using the Koenig--Yang correspondences in silting theory.
First, we introduce an equivalence relation on $K_0(\proj A)_\R$
called TF equivalence by using numerical torsion pairs
of Baumann--Kamnitzer--Tingley.
Second, we show that the open cone in $K_0(\proj A)_\R$ spanned by 
the g-vectors of each 2-term silting object gives a TF equivalence class, and 
this gives a one-to-one correspondence between the basic 2-term silting objects and 
the TF equivalence classes of full dimension.
Finally, we determine the wall-chamber structure of $K_0(\proj A)_\R$
in the case that $A$ is a path algebra of an acyclic quiver.
\end{abstract}

\maketitle

\tableofcontents

\section{Introduction}

It is well-known that 
projective modules and simple modules are fundamental and important objects in 
the representation theory of a ring $A$.
This paper is devoted to study mutual relationship between these modules
in the case that $A$ is a finite-dimensional algebra over a field $K$.
In this setting, there are only finitely many isomorphism classes 
$S_1,\ldots,S_n$ of simple $A$-modules 
in the category $\mod A$ of finite-dimensional $A$-modules.
They bijectively correspond to the isomorphism classes
$P_1, \ldots, P_n$ of indecomposable projective $A$-modules
in the category $\proj A$ of finitely generated projective $A$-modules 
via taking the projective covers $P_i \to S_i$.
Moreover, $\Hom_A(P_i,S_j) \ne 0$ holds if and only if $i=j$.

Such relationship between projective modules and simple modules has been 
extended to derived categories.
As a generalization of progenerators and classical tilting modules, 
Keller--Vossieck \cite{KV} introduced \textit{silting objects} (Definition \ref{Def_silt})
of the perfect derived category $\sfK^\rmb(\proj A)$.
Then, Koenig--Yang \cite{KY} found that the silting objects 
have one-to-one correspondences with many important notions,
including the bounded \textit{t-structures} with length heart (Definition \ref{Def_tstr})
and the \textit{simple-minded collections} (Definition \ref{Def_smc}) 
in the bounded derived category $\sfD^\rmb(\mod A)$.
These bijections are collectively called the \textit{Koenig--Yang correspondences},
and have been developed by many authors such as \cite{BY, AIR, Asai, MS}.

The Koenig--Yang correspondences can be studied from the point of view of    
the \textit{Grothendieck groups} $K_0(\proj A)$ and $K_0(\mod A)$
and the Euler form.
The \textit{Euler form} is 
a $\Z$-bilinear form 
\begin{align*}
\ang{?}{!} \colon K_0(\proj A) \times K_0(\mod A) \to \Z
\end{align*} 
defined by
\begin{align*}
\ang{T}{X} :=\sum_{k \in \Z}(-1)^k \dim_K \Hom_{\sfD^\rmb(\mod A)}(T,X[k]).
\end{align*}
for every $T \in \sfK^\rmb(\proj A)$ and
$X \in \sfD^\rmb(\mod A)$.
With respect to the Euler form, the families $(P_i)_{i=1}^n$ and $(S_i)_{i=1}^n$ are dual bases
of $K_0(\proj A)$ and $K_0(\mod A)$ in the following sense:
\begin{align*}
\ang{P_i}{S_j} = \begin{cases}
\dim_K \End_{\sfD^\rmb(\mod A)}(S_j) & (i=j) \\
0 & (i \ne j)
\end{cases}.
\end{align*}

For example, \cite{KR, DF, DIJ} studied 
the \textit{g-vector} $[U]=[U^0]-[U^{-1}] \in K_0(\proj A)$ 
of a 2-term presilting object $U=(U^{-1} \to U^0)$ in $\sfK^\rmb(\proj A)$
by using the presentation space $\Hom_A(U^{-1},U^0)$.
Moreover, Aihara--Iyama \cite{AI} showed that
the g-vectors of the indecomposable direct summands of every silting object 
give a $\Z$-basis of the Grothendieck group $K_0(\proj A)$,
and Koenig--Yang \cite{KY} showed that this basis is dual to 
the $\Z$-basis of $K_0(\mod A)$ given by the corresponding simple-minded collection.

Each $\theta \in K_0(\proj A)$ gives a \textit{stability condition} for modules in $\mod A$
in the sense of King \cite{King} via the Euler form.
Stability conditions play an important role in many aspects,
including the construction of moduli spaces of modules in geometric invariant theory \cite{King},
the detailed study of crystal bases of quantum groups 
from preprojective algebras \cite{BKT},
and the investigation of scattering diagrams and quivers with potentials 
in cluster theory \cite{Bridgeland}.

In this paper, we consider the real Grothendieck groups 
\begin{align*}
K_0(\proj A)_\R:=K_0(\proj A) \otimes_\Z \R
\quad \text{and} \quad K_0(\mod A)_\R:=K_0(\mod A) \otimes_\Z \R.
\end{align*}
The stability condition given by each $\theta \in K_0(\proj A)_\R$ is 
nothing but a collection of linear inequalities; 
namely, a module $M \in \mod A$ is said to be \textit{$\theta$-semistable}
if $\theta(M)=0$ and $\theta(X) \ge 0$ for all quotient modules $X$ of $M$.
The subcategory $\calW_\theta \subset \mod A$ of $\theta$-semistable modules 
is a \textit{wide subcategory} of $\mod A$, 
that is, a subcategory closed under taking kernels, cokernels and extensions.
In particular, $\calW_\theta$ is an abelian length category,
and the simple objects in $\calW_\theta$ are precisely the $\theta$-stable modules.
Every simple object $S$ in $\calW_\theta$ is a \textit{brick},
that is, the endomorphism ring $\End_A(S)$ is a division $K$-algebra.

Each nonzero module $M$ gives the rational polyhedral cone 
$\Theta_M \subset K_0(\proj A)_\R$ called the \textit{wall}
consisting of $\theta \in K_0(\proj A)_\R$ such that $M$ is $\theta$-semistable.
The subsets $\Theta_M$ for all $M$ give a \textit{wall-chamber structure} 
in $K_0(\proj A)_\R$ studied in \cite{BST,Bridgeland}.

In this paper, we study the wall-chamber structure of $K_0(\proj A)_\R$
by using the two \textit{numerical torsion pairs} 
$(\ovcalT_\theta,\calF_\theta)$ and $(\calT_\theta,\ovcalF_\theta)$
for each $\theta \in K_0(\proj A)_\R$ introduced by \cite{BKT},
which are defined by linear inequalities in a similar way to stability conditions.

Our first aim in this paper 
is to investigate the wall-chamber structure of $K_0(\proj A)_\R$
via the numerical torsion pairs.
For this purpose,
we define an equivalence relation on $K_0(\proj A)_\R$ as follows:
we say that $\theta$ and $\theta'$ are \textit{TF equivalent}
if $(\ovcalT_\theta,\calF_\theta)=(\ovcalT_{\theta'},\calF_{\theta'})$ and 
$(\calT_\theta,\ovcalF_\theta)=(\calT_{\theta'},\ovcalF_{\theta'})$.
It is easily seen that any TF equivalence class is convex in $K_0(\proj A)_\R$.
The following first main result of this paper characterizes 
the TF equivalence classes in terms of the walls $\Theta_M$.

\begin{Thm}[Theorem \ref{Thm_W_constant}]\label{Thm_W_constant_intro}
Let $\theta,\theta' \in K_0(\proj A)_\R$ be distinct elements.
Then the following conditions are equivalent.
\begin{itemize}
\item[(a)]
The elements $\theta$ and $\theta'$ are TF equivalent.
\item[(b)]
For any $\theta'' \in [\theta,\theta']$, 
the $\theta''$-semistable subcategory $\calW_{\theta''}$ is constant.
\item[(c)]
There does not exist a brick $S$ such that 
$[\theta,\theta'] \cap \Theta_S$ has exactly one element.
\end{itemize}
\end{Thm}

Theorem \ref{Thm_W_constant_intro} gives the following easy observation.
We write $\Wall:=\bigcup_{M \in \mod A \setminus \{0\}} \Theta_M$,
and consider the complement $K_0(\proj A)_\R \setminus \overline{\Wall}$
of the closure $\overline{\Wall}$.
We call each connected component of $K_0(\proj A)_\R \setminus \overline{\Wall}$
a \textit{chamber} of the wall-chamber structure of $K_0(\proj A)_\R$,
and define $\Chamber(A)$ as the set of chambers.
In view of Theorem \ref{Thm_W_constant_intro}, we know that
any chamber is contained in some TF equivalence class of dimension $n$.
Conversely, we can check that,
if $[\theta]$ is a TF equivalence class of dimension $n$, 
then $[\theta]^\circ$ is a chamber.
Let $\TF_n(A)$ denote the set of TF equivalence classes of dimension $n$
in $K_0(\proj A)_\R$.
Then, we get the following result.

\begin{Cor}[Corollary \ref{Cor_TF_chamber}]\label{Cor_TF_chamber_intro}
There exists a bijection
\begin{align*}
\Chamber(A) &\to \TF_n(A) \\
C &\mapsto \textup{(the TF equivalence class containing $C$)}.
\end{align*}
\end{Cor}

Our next aim is to prove that a chamber is actually a TF equivalent class;
in other words, all TF equivalent classes of dimension $n$ are open.
For this purpose, we use the Koenig--Yang correspondences.
To each 2-term presilting object $U$ in $\sfK^\rmb(\proj A)_\R$,
we associate the cones $C(U)$ and $C^+(U)$ in $K_0(\proj A)_\R$
given by
\begin{align*}
C(U)   &:= \{ a_1[U_1]+\cdots+a_m[U_m] \mid a_1,\ldots,a_m \in \R_{\ge 0} \}, \\
C^+(U) &:= \{ a_1[U_1]+\cdots+a_m[U_m] \mid a_1,\ldots,a_m \in \R_{> 0} \},
\end{align*}
following Demonet--Iyama--Jasso \cite{DIJ}.
The following second main result of this paper, 
based on Yurikusa's work \cite{Yurikusa}, 
shows that each 2-term presilting object $U$ gives a TF equivalence class $C^+(U)$.
\begin{Thm}[Theorem \ref{Thm_presilt_TF}]
Let $U \in \twopresilt A$.
Then, the subset $C^+(U) \subset K_0(\proj A)_\R$ is a TF equivalence class,
and any $\theta \in C^+(U)$ satisfies
\begin{align*}
(\ovcalT_\theta, \calF_\theta) = 
({{}^\perp H^{-1}(\nu U)}, \Sub H^{-1}(\nu U)), \quad
(\calT_\theta, \ovcalF_\theta) = 
(\Fac H^0(U), H^0(U)^\perp).
\end{align*}
\end{Thm}

In particular, the correspondence $U \mapsto C^+(U)$ gives an injection from
the set $\twopresilt A$ of basic 2-term presilting objects in $\sfK^\rmb(\proj A)$
to the set of TF equivalence classes.
Moreover, if $T$ belongs to the set $\twosilt A$ of basic 2-term silting objects 
in $\sfK^\rmb(\proj A)$,
then the TF equivalence class $C^+(T)$ is an open set of $K_0(\proj A)_\R$;
hence $C^+(T)$ is a chamber by Corollary \ref{Cor_TF_chamber_intro}.
This recovers the injection $\twosilt A \to \Chamber(A)$ given by 
Br\"{u}stle--Smith--Treffinger \cite{BST}.
We have shown that this map is actually a bijection as follows.
We set
\begin{align*}
K_0(\proj A)_\Q &:= \{ a_1[P_1]+\cdots+a_n[P_n] \mid a_1,\ldots,a_n \in \Q \}, \\
C^+(T)_\Q &:= C^+(T) \cap K_0(\proj A)_\Q.
\end{align*}

\begin{Thm}[Theorem \ref{Thm_chamber_cone}]\label{Thm_chamber_cone_intro}
We have equations
\begin{align*}
\coprod_{T \in \twosilt A} C^+(T) &= K_0(\proj A) \setminus \overline{\Wall}, \\
\coprod_{T \in \twosilt A} C^+(T)_\Q &= K_0(\proj A)_\Q \setminus \Wall.
\end{align*}
In particular, there exists a bijection
\begin{align*}
\twosilt A \to \Chamber(A), \qquad T \mapsto C^+(T). 
\end{align*}
Therefore, all chambers are TF equivalence classes, so $\TF_n(A)=\Chamber(A)$.
\end{Thm}


In Section 4,
we describe how TF equivalence classes change under \textit{$\tau$-tilting reduction}
introduced in \cite{Jasso, DIRRT}. 
For a fixed 2-term presilting object $U$,
they constructed a certain finite-dimensional algebra $B$ 
and bijections between the subset $\twosilt_U A \subset \twosilt A$ of
basic 2-term presilting objects containing $U$ as a direct summand
and the set $\twosilt B$.
More explicitly, they obtained a bijection
\begin{align*}
\red:=\Hom_{\sfK^\rmb(\proj A)}(T,?)/[U] \colon \twosilt_U A \to \twosilt B,
\end{align*}
where $T$ is the \textit{Bongartz completion} of $U$ 
and $B:=\End_{\sfK^\rmb(\proj A)}(T)/[U]$.
It is easily extended to a bijection 
\begin{align*}
\red:=\Hom_{\sfK^\rmb(\proj A)}(T,?)/[U] \colon \twopresilt_U A \to \twopresilt B,
\end{align*}
between the 2-term presilting objects.
They also showed that $\calW_U:={{}^\perp H^{-1}(\nu U)} \cap H^0(U)^\perp$
is a wide subcategory of $\mod A$, and that there exists
a category equivalence 
\begin{align*}
\phi:=\Hom_A(T,?) \colon \calW_U \to \mod B.
\end{align*}

We would like to know the wall-chamber structure of $K_0(\proj B)_\R$
in this situation.
For this purpose, we define an open neighborhood $N_U$ of
$[U] \in K_0(\proj A)_\R$ by
\begin{align*}
N_U:=\{ \theta \in K_0(\proj A)_\R \mid
H^0(U) \in \calT_\theta, \ H^{-1}(\nu U) \in \calF_\theta \}.
\end{align*}
Clearly, $N_U$ is a union of some TF equivalence classes in $K_0(\proj A)_\R$.
We prove that 
the local wall-chamber structure of $N_U \subset K_0(\proj A)_\R$ 
around $[U]$ recovers
the whole wall-chamber structure of $K_0(\proj B)_\R$ 
via the linear projection $\pi \colon K_0(\proj A)_\R \to K_0(\proj B)_\R$
given by
\begin{align*}
\pi(\theta):=\sum_{i=1}^m \frac{\theta(X_i)}{d_i}[P^B_i],
\end{align*}
where $X_1,X_2,\ldots,X_m$ are the simple objects of $\calW_U$,
$d_i:=\dim_K \End_A(X_i)$, and $P_i^B$ is the projective cover of 
the simple $B$-module $\phi(X_i)$ for each $i$.

\begin{Thm}[Theorem \ref{Thm_N_U_TF_union}]
Let $U \in \twopresilt A$. Then, we have the following properties.
\begin{itemize}
\item[(1)]
For any $\theta \in N_U$ and $M \in \calW_U$,
the wall $\Theta_{\phi(M)}$ coincides with $\pi(\Theta_M \cap N_U)$.
\item[(2)]
The linear map $\pi$ induces a bijection 
\begin{align*}
\{\textup{TF equivalence classes in $N_U$}\} & \to 
\{\textup{TF equivalence classes in $K_0(\proj B)_\R$}\}, \\
[\theta] &\mapsto \pi([\theta]).
\end{align*}
\item[(3)]
We have the following commutative diagram:
\begin{align*}
\begin{xy}
(  0,  8) *+{\twopresilt_U A}="1",
( 80,  8) *+{\twopresilt B}="2",
(  0, -8) *+{\{\textup{TF equivalence classes in $N_U$}\}}="3",
( 80, -8) *+{\{\textup{TF equivalence classes in $K_0(\proj B)_\R$}\}}="4".
\ar_{C^+} "1";"3"
\ar_{C^+} "2";"4"
\ar^{\red}_{\cong} "1";"2"
\ar^{\pi}_{\cong} "3";"4"
\end{xy}.
\end{align*}
\end{itemize}
\end{Thm}

As an application of this theorem, we give a simple proof of
the following characterization of $\tau$-tilting finiteness by the cones $C(T)$
for 2-term silting objects $T$.
Recall that $A$ is said to be \textit{$\tau$-tilting finite} 
if $\#(\twosilt A)<\infty$ \cite{DIJ,AIR}.

\begin{Thm}[Theorem \ref{Thm_finite_cover_eq}]
The algebra $A$ is $\tau$-tilting finite 
if and only if $K_0(\proj A)_\R=\bigcup_{T \in \twosilt A} C(T)$.
\end{Thm}

Note that the ``only if'' part follows from \cite{DIJ}.
The ``if'' part was conjectured by Demonet \cite{Demonet}
and a different proof was given by Zimmermann--Zvonareva \cite{ZZ}.

Finally, we give a combinatorial method 
to obtain the wall-chamber structure of $K_0(\proj A)_\R$
in the case that $A$ is the path algebra of an acyclic quiver $Q$
over an algebraically closed field $K$.
For each nonzero dimension vector $\vecd$,
there exists a module $M$
which gives the largest wall $\Theta_M$ with respect to inclusion 
among all modules whose dimension vectors are $\vecd$ \cite{Schofield}.
We write this largest wall $\Theta_{\vecd}$.
Then, the wall $\Theta_{\vecd}$ can be determined inductively in the following 
way.

\begin{Thm}[Theorem \ref{Thm_wall_path}]
Let $Q$ be an acyclic quiver and 
$\vecd=(d_i)_{i \in Q_0} \in (\Z_{\ge 0})^{Q_0}$ 
be a nonzero dimension vector,
and set $\supp \vecd:=\{i \in Q_0 \mid d_i \ne 0\}$.
Then, $\Theta_{\vecd}$ is given as follows.
\begin{itemize}
\item[(1)]
If $\# \supp \vecd = 1$ and $k \in \supp {\vecd}$,
then $\Theta_{\vecd}=\bigoplus_{i \ne k} \R[P_i]$.
\item[(2)]
Assume that $\# \supp \vecd = 2$
and that the full subquiver of $Q$ corresponding to $\supp \vecd \subset Q$ is 
\begin{align*}
\begin{xy}
( 0, 0)*+{k}="1",
(20, 0)*+{l}="2",
(10, 0)*+{\vdots},
\ar@<4mm> "1";"2"
\ar@<2mm> "1";"2"
\ar@<-4mm> "1";"2"
\end{xy}
\quad \textup{($m$ arrows)}
\end{align*} 
with $k,l \in \supp \vecd$ and $m \in \Z_{\ge 0}$.
We define $a,b \in \Z_{\ge 0}$ by $a:=d_k/\gcd(d_k,d_l)$ and $b:=d_l/\gcd(d_k,d_l)$.
Then, 
\begin{align*}
\Theta_{\vecd}=\begin{cases}
(\bigoplus_{i \ne k,l} \R[P_i]) \oplus \R_{\ge 0}(b[P_k]-a[P_l]) 
& (a^2+b^2-mab \le 1) \\
(\bigoplus_{i \ne k,l} \R[P_i])
& (\textup{otherwise})
\end{cases}.
\end{align*}
\item[(3)]
If $\supp \vecd \ge 3$, then 
$\Theta_{\vecd}$ is the smallest polyhedral cone of 
$K_0(\proj A)_\R$ containing
\begin{align*}
\bigcup_{0 < \vecc < \vecd} 
(\Theta_{\vecc} \cap \Theta_{\vecd-\vecc}).
\end{align*}
\end{itemize}
\end{Thm}
As an example of the theorem above,
we give the wall-chamber structure of $K_0(\proj A)_\R$
in the case that $Q$ is the wild quiver $1 \rightrightarrows 2 \to 3$
in Example \ref{Ex_pptx1}.
Since $A$ is hereditary,
there exists a canonical isomorphism $K_0(\proj A)_\R \cong K_0(\mod A)_\R$,
and under this identification,
the chambers in $K_0(\proj A)_\R$ coincide 
with the diagram in $K_0(\mod A)_\R$ 
given in \cite[Example 11.3.9]{DW}.

\subsection{Notation}

Throughout this paper, $K$ is a field and $A$ is a finite-dimensional $K$-algebra.
Unless otherwise stated, all algebras and modules are finite-dimensional.

We set $\proj A$ as the category of finitely generated projective right $A$-modules,
and $P_1,\ldots,P_n$ as all the non-isomorphic indecomposable projective modules 
in $\proj A$.
Similarly, we write $\mod A$ for the category of finite-dimensional right $A$-modules,
and let $S_1,\ldots,S_n$ be all the non-isomorphic simple modules in $\mod A$.
We may additionally assume 
that $S_i$ is the top of $P_i$ for each $i \in \{1,\ldots,n\}$.

The symbol $\sfK^\rmb(\proj A)$ denotes the homotopy category of 
the bounded complex category of $\proj A$,
and $\sfD^\rmb(\mod A)$ stands for the derived category of 
the bounded complex category of $\mod A$.
Both categories are triangulated categories,
and their shifts are denoted by $[1]$.

Any subcategory appearing in this paper is a full subcategory,
and is assumed to be closed under isomorphism classes.


\section{TF equivalence on stability conditions}

\subsection{The wall-chamber structures}

We start by recalling the definition of Grothendieck groups.
Let $\calC$ be an exact category or a triangulated category,
then the \textit{Grothendieck group} $K_0(\calC)$ is the quotient group 
of the free abelian group on the set of isomorphism classes $[X]$ of $\calC$
by the relations $[X]-[Y]+[Z]=0$ 
for all admissible short exact sequences $0 \to X \to Y \to Z \to 0$
or all triangles $X \to Y \to Z \to X[1]$ in $\calC$.

It is well-known that the Grothendieck group $K_0(\proj A)$ has a $\Z$-basis 
$(P_i)_{i=1}^n$ given by all the non-isomorphic indecomposable projective modules,
and that it is canonically isomorphic to $K_0(\sfK^\rmb(\proj A))$.
Similarly, $K_0(\mod A)$ is also a free abelian group of rank $n$, and
the family $(S_i)_{i=1}^n$ of all the non-isomorphic simple modules is a $\Z$-basis 
of $K_0(\mod A)$.
The Grothendieck group $K_0(\mod A)$ can be canonically identified with 
$K_0(\sfD^\rmb(\mod A))$; see \cite{Happel} for details.

For these Grothendieck groups $K_0(\proj A)$ and $K_0(\mod A)$,
we consider a non-degenerate $\Z$-bilinear form 
$\ang{?}{!} \colon K_0(\proj A) \times K_0(\mod A) \to \Z$ 
called the \textit{Euler form} defined by 
\begin{align*}
\ang{T}{X} :=\sum_{k \in \Z}(-1)^k \dim_K \Hom_{\sfD^\rmb(\mod A)}(T,X[k])
\end{align*}
for $T \in \sfK^\rmb(\proj A)$ and $X \in \sfD^\rmb(\mod A)$.
The families $(P_i)_{i=1}^n$ and $(S_i)_{i=1}^n$ 
give dual bases of $K_0(\proj A)$ and $K_0(\mod A)$
with respect to the Euler form in the following sense:
\begin{align*}
\ang{P_i}{S_j} = \begin{cases}
\dim_K \End_{\sfD^\rmb(\mod A)}(S_j) & (i=j) \\
0 & (i \ne j)
\end{cases}.
\end{align*}

In this paper, we consider the real Grothendieck groups 
\begin{align*}
K_0(\proj A)_\R:=K_0(\proj A) \otimes_\Z \R \quad \text{and} \quad
K_0(\mod A)_\R:=K_0(\mod A) \otimes_\Z \R.
\end{align*}
Then, they are identified with the Euclidean space $\R^n$
as topological spaces and vector spaces.
The Euler form is
naturally extended to an $\R$-bilinear form
$\ang{?}{!} \colon K_0(\proj A)_\R \times K_0(\mod A)_\R \to \R$.
We regard each $\theta \in K_0(\proj A)_\R$ as an $\R$-linear form 
$\ang{\theta}{?} \colon K_0(\mod A)_\R \to \R$;
in other words, we set $\theta(M):=\ang{\theta}{M}$.
For each $\theta \in K_0(\proj A)_\R$,
King \cite{King} associated a \textit{stability condition} as follows.

\begin{Def}
\cite[Definition 1.1]{King}
Let $\theta \in K_0(\proj A)_\R$.
\begin{itemize}
\item[(1)]
A module $M \in \mod A$ is said to be $\theta$-\textit{semistable} if 
\begin{itemize}
\item
$\theta(M)=0$, and 
\item
for any quotient module $X$ of $M$, we have $\theta(X) \ge 0$.
\end{itemize}
We define the $\theta$-\textit{semistable subcategory} $\calW_\theta$ 
as the full subcategory consisting of all the $\theta$-semistable modules in $\mod A$.
\item[(2)]
A module $M \in \mod A$ is said to be $\theta$-\textit{stable} if 
\begin{itemize}
\item
$M \ne 0$,
\item
$\theta(M)=0$, and 
\item
for any nonzero proper quotient module $X$ of $M$, 
we have $\theta(X) > 0$.
\end{itemize}
\end{itemize}
\end{Def}

The $\theta$-semistable subcategory $\calW_\theta$ 
is a \textit{wide subcategory} of $\mod A$,
that is, a full subcategory closed under kernels, cokernels, and 
extensions in $\mod A$.
In particular, $\calW_\theta$ is an abelian category,
so all its simple objects are bricks.
Here, a module $S \in \mod A$ is called a \textit{brick} if
its endomorphism ring $\End_A(S)$ is a division ring.
We write $\brick A$ for the set of isomorphism classes of bricks in $\mod A$.
By definition, we obtain the following property.

\begin{Lem}\label{Lem_simple_stable}
Let $\theta \in K_0(\proj A)_\R$ and 
$M \in \calW_\theta$, then $M$ is a simple object in $\calW_\theta$ 
if and only if $M$ is $\theta$-stable.
\end{Lem}

To investigate semistable subcateories, 
we associate a wall for each nonzero module in $\mod A$
as Br\"{u}stle--Smith--Treffinger \cite[Definition 3.2]{BST} and 
Bridgeland \cite[Definition 6.1]{Bridgeland}.

\begin{Def}
For any nonzero module $M \in \mod A \setminus \{ 0 \}$, we set 
\begin{align*}
\Theta_M:=\{ \theta \in K_0(\proj A)_\R \mid M \in \calW_\theta \},
\end{align*}
and call $\Theta_M$ the \textit{wall} associated to the module $M$.
\end{Def}

These walls define a \textit{wall-chamber structure} of $K_0(\proj A)_\R$.
Clearly, 
we have $\Theta_{M_1 \oplus M_2}=\Theta_{M_1} \cap \Theta_{M_2}$ 
for any $M_1, M_2 \in \mod A \setminus \{ 0 \}$,
so we sometimes consider the walls only for indecomposable modules.
We here give an easy example.

\begin{Ex}
Let $A$ be the path algebra $K(1 \to 2)$.
The indecomposable $A$-modules are $S_1,S_2,P_1$,
and the corresponding walls are
$\Theta_{S_1}=\R[P_2]$, $\Theta_{S_2}=\R[P_1]$
and $\Theta_{P_1}=\R_{\ge 0}([P_1]-[P_2])$,
since there exists a short exact sequence $0 \to S_2 \to P_1 \to S_1 \to 0$.
These walls are depicted as follows:
\begin{align*}
\includegraphics[width=5.7cm]{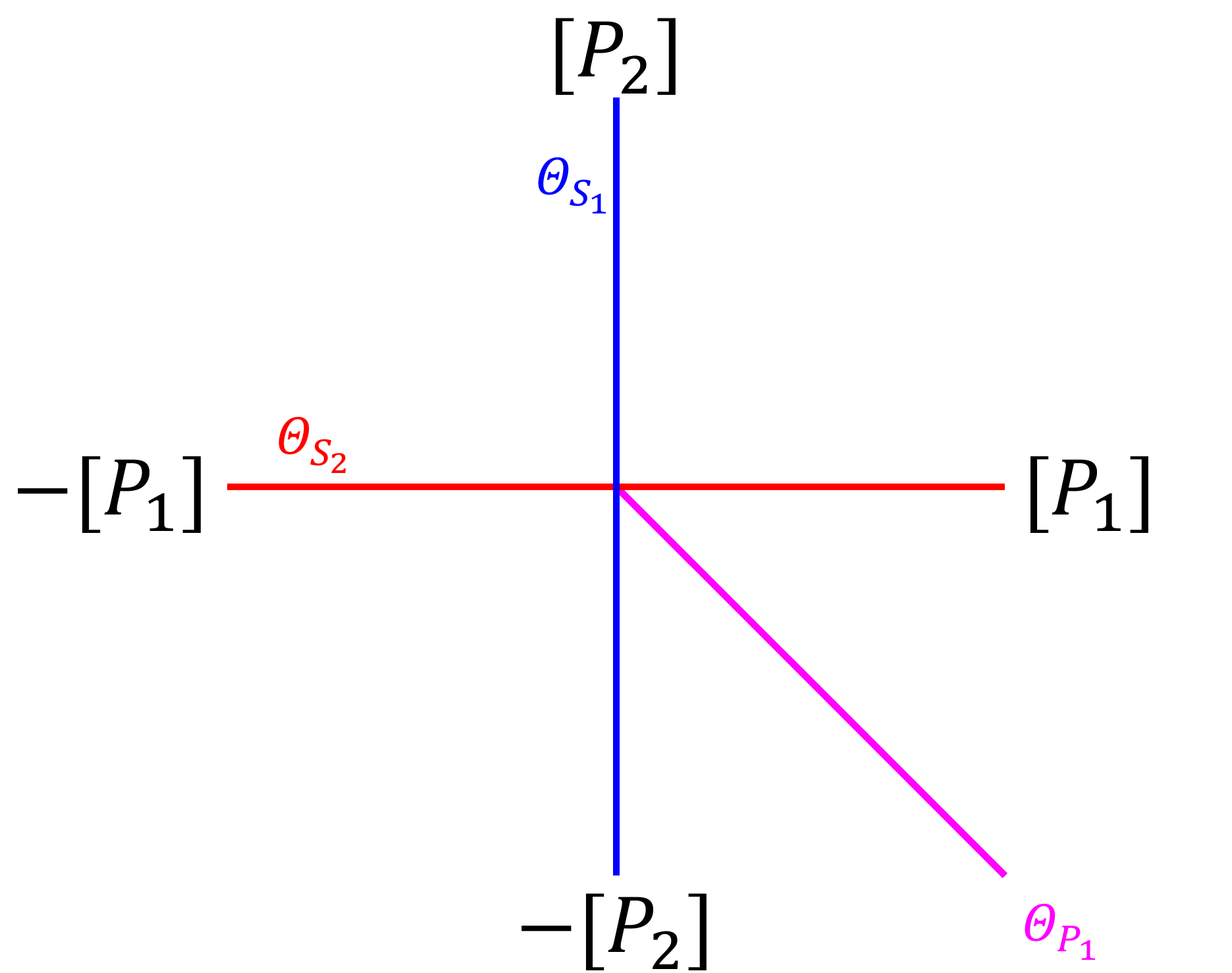}
\end{align*}
\end{Ex}

To investigate the walls $\Theta_M$ more geometrically,
we here cite some basic notions and properties on rational polyhedral cones
in a Euclidean space from \cite[Section 1.2]{Fulton}.

Let $F$ be a free abelian group of finite rank,
and set $F^*:=\Hom_\Z(F,\Z)$.
Then, we have two finite-dimensional $\R$-vector spaces
$V:=F \otimes_\Z \R$ and $V^*:=F^* \otimes_\Z \R \cong \Hom_\R(V,\R)$.
A subset $D \subset V$ is called a \textit{polyhedral cone}
if there exist finitely many elements $v_1,\ldots,v_m \in V$ 
such that 
\begin{align*}
D=\left\{ \sum_{i=1}^m r_i v_i \mid r_i \in \R_{\ge 0} \right\}.
\end{align*}
A polyhedral cone $D$ is said to be \textit{rational} 
if we can take $v_1,\ldots,v_m$ above 
so that $v_1,\ldots,v_m \in F \otimes_\Z \Q$.

For a polyhedral cone $D$ in $V$, 
we define the \textit{dual} $D^\vee \subset V^*$ of $D$ by
\begin{align*}
D^\vee:=\{ u \in V^* \mid \text{for all $v \in D$, $\ang{u}{v} \ge 0$} \},
\end{align*}
where $\ang{u}{v}:=u(v)$.
Then, $D^\vee$ is a polyhedral cone in $V^*$, that is,
there exist finitely many elements $u_1,\ldots,u_m \in V^*$ 
such that
\begin{align*}
D^\vee=\left\{ \sum_{i=1}^m r_i u_i \mid r_i \in \R_{\ge 0} \right\}.
\end{align*}
Moreover, if $D$ is rational, then $D^\vee$ is rational.

We can consider the dual polyhedral cone $C^\vee$ in $V$ 
of a polyhedral cone $C$ in $V^*$ in a similar way,
and then, $(C^\vee)^\vee$ coincides with $C$.

Let $C$ be a polyhedral cone in $V^*$.
A subset $C' \subset C$ is called a \textit{face} if 
there exists some $v \in C^\vee$ such that $C'=C \cap \Ker \ang{?}{v}$,
or equivalently, if
$C'$ admits finitely many elements $v_1,v_2,\ldots,v_m \in C^\vee$
which satisfy $C'=C \cap \left(\bigcap_{i=1}^m \Ker \ang{?}{v_i}\right)$.
Any face of a (rational) polyhedral cone is a (rational) polyhedral cone again.
We define the \textit{dimension} $\dim C$ of the polyhedral cone $C$
as the dimension $\dim_\R (\R \cdot C)$ 
of the $\R$-vector subspace $\R \cdot C \subset V^*$ spanned by $C$.
We say that a polyhedral cone $C$ is 
\textit{strongly convex} if the vector space $C \cap (-C)$ is $\{0\}$.

By setting $F:=K_0(\mod A)$, 
we get $F^* \cong K_0(\proj A)$ via the Euler form.
For each $M \in \mod A \setminus \{0\}$,
consider the rational polyhedral cone $D_M$ in $K_0(\mod A)_\R$ generated by the set 
\begin{align*}
\{[X] \mid \text{$X$ is a quotient module of $M$} \} \cup \{ -[M] \},
\end{align*}
then the wall $\Theta_M$ coincides with the dual $(D_M)^\vee$,
so $\Theta_M$ is a rational polyhedral cone in $K_0(\proj A)_\R$.
A finite set $\{X_1,X_2,\ldots,X_m\}$ of quotient modules of $M$
gives a face 
$\Theta_M \cap \left( \bigcap_{i=1}^m \Ker \ang{?}{X_i} \right)$,
and we can check that all faces of $\Theta_M$ are obtained in this way.
Since $\Theta_M \subset \Ker \ang{?}{M}$,
we have $\dim \Theta_M \le n-1$.

From now on, we will characterize some conditions on $\Theta_M$ 
as a polyhedral cone 
in terms of representation theoritic properties of $M$.

We first consider the question when $\Theta_M$ is strongly convex.
The answer is given by the sincerity of the module $M$.
We say that $M \in \mod A$ is \textit{sincere} 
if $\supp M=\{1,2,\ldots,n\}$, 
where we set
\begin{align*}
\supp M:=\{ i \in \{1,2,\ldots,n\} \mid 
\text{$S_i$ appears in a composition series of $M$ in $\mod A$} \}.
\end{align*}

\begin{Lem}\label{Lem_strongly_convex}
Let $M \in \mod A \setminus \{ 0 \}$, and set 
$H_1:=\bigoplus_{i \in \supp M} \R[S_i]$ and 
$H_2:=\bigoplus_{j \notin \supp M} \R[S_j]$.
Then, we have the following assertions.
\begin{itemize}
\item[(1)]
We have $\Theta_M \cap (-\Theta_M)=H_2$.
\item[(2)]
The wall $\Theta_M$ is strongly convex if and only if $M$ is sincere.
\item[(3)]
The wall $\Theta_M$ coincides with 
$(\Theta_M \cap H_1) \oplus H_2$,
and $\Theta_M \cap H_1$ is a strongly convex polyhedral cone in $H_1$.
\end{itemize}
\end{Lem}

\begin{proof}
(1)
We first show $\Theta_M \cap (-\Theta_M) \subset H_2$.
Assume $\theta \in \Theta_M \cap (-\Theta_M)$,
and take a composition series 
$0 = M_0 \subset M_1 \subset \cdots \subset M_l = M$ in $\mod A$.
Since $\theta \in \Theta_M \cap (-\Theta_M)$,
we have $\theta(M_k) \le 0$ and $-\theta(M_k) \le 0$ for all $k$,
so $\theta(M_k)=0$ for any $k$.
Therefore, $\theta(M_k/M_{k-1})=0$ holds for all $k \in \{1,\ldots,l\}$,
which clearly implies $\theta \in H_2$.
The converse inclusion is obvious.

(2)
This straightly follows from (1).

(3)
The first statement is clear.
We take the idempotent $e \in A$ such that
$S_i e = 0$ if and only if $i \in \supp M$,
then $M$ is a sincere $A/\langle e \rangle$-module.
We define the wall $\Theta_M^{A/\langle e \rangle}$ 
associated to $M \in \mod A/\langle e \rangle$ 
in $K_0(\mod A/\langle e \rangle)_\R$,
which is strongly convex by (2).
Under the canonical inclusion $K_0(\mod A/\langle e \rangle)_\R
\to K_0(\mod A)_\R$, the image of
$\Theta_M^{A/\langle e \rangle} \subset K_0(\mod A/\langle e \rangle)_\R$ 
is $\Theta_M \cap H_1$.
Thus, $\Theta_M \cap H_1$ is a strongly convex polyhedral cone in $H_1$.
\end{proof}

We next focus on the faces of the polyhedral cone 
$\Theta_M$ in $K_0(\proj A)_\R$.
For $M \in \mod A \setminus \{ 0 \}$ and $\theta \in \Theta_M$,
we define 
\begin{align*}
\supp_\theta M:=\{
S \in \brick A \mid \text{$S$ is a simple object 
appearing in a composition series of $M$ in $\calW_\theta$} \}.
\end{align*}

\begin{Lem}\label{Lem_face}
Let $M \in \mod A \setminus \{ 0 \}$ and $\theta \in \Theta_M$,
and set $H:=\Ker \ang{?}{\supp_\theta M} \subset K_0(\proj A)_\R$.
Then, we have the following properties.
\begin{itemize}
\item[(1)]
The element $\theta$ belongs to the interior of $\Theta_M \cap H$ in $H$,
and $\Theta_M \cap H$ is the smallest face of $\Theta_M$ containing $\theta$.
\item[(2)]
Let $\theta' \in \Theta_M$. 
Then, $\theta' \in \Theta_M \cap H$ holds 
if and only if $\supp_\theta M \subset \calW_{\theta'}$.
\item[(3)]
Let $\theta' \in \Theta_M$ and set $H':=\Ker \ang{?}{\supp_{\theta'} M}$.
Then, $H=H'$ holds if and only if $\supp_\theta M = \supp_{\theta'} M$.
\end{itemize}
\end{Lem}

\begin{proof}
(1)
Clearly, we have $\theta \in \Theta_M \cap H$.
For each $X \in \supp_\theta M$, 
there exists an open subset $N_X \subset K_0(\proj A)_\R$ 
such that $\theta \in N_X$ by Lemma \ref{Lem_simple_stable}. 
Then, $\theta \in (\bigcap_{X \in \supp_\theta M} N_X) \cap \Theta_M \cap H$ holds.
Since $\supp_\theta M$ is a finite set,
$\bigcap_{X \in \supp_\theta M} N_X$ is an open subset of $K_0(\proj A)_\R$.
Therefore, $\theta$ belongs to the interior of $\Theta_M \cap H$ in $H$.

Next, we show that $\Theta_M \cap H$ is a face of $\Theta_M$.
Take a composition series $0 \subset M_0 \subset M_1 \subset \cdots \subset M_l = M$
in $\calW_\theta$,
then $\supp_\theta M = \{M_i/M_{i-1} \mid i \in \{1,2,\ldots,l\} \}$.
Thus, for each $\theta' \in \Theta_M$,
the condition $\theta(\supp_\theta M)=0$ holds if and only if
$\theta(M/M_i)=0$ for all $i \in \{0,1,\ldots,l-1\}$,
so we get $\Theta_M \cap H = \bigcap_{i=0}^{l-1} (\Theta_M \cap \Ker \ang{?}{M/M_i})$.
Since $M/M_i$ is a quotient module of $M$,
the subset $\Theta_M \cap H$ is a face of $\Theta_M$.

These two facts yield that
$\Theta_X \cap \Ker \ang{?}{\supp_\theta M}$ is 
the smallest face containing $\theta$.

(2)
The ``if'' part is obvious, so we consider the ``only if'' part.
Let $\theta' \in \Theta_M \cap H$ and $X \in \supp_\theta M$.
It suffices to show $X \in \calW_{\theta'}$.
First, $\theta'(X)=0$ follows from the definition of $H$.
Next, assume that $Y$ is a quotient module of $X$.
Consider the composition series in the proof of (1), 
then we can take $i \in \{1,2,\ldots,l\}$ such that $X \cong M_i/M_{i-1}$.
It follows that there exists a quotient module $N$ of $M$ admitting
a short exact sequence $0 \to Y \to N \to M/M_i \to 0$.
Since $\theta' \in \Theta_M$, we have $\theta'(N) \ge 0$.
Moreover, $\theta'(M/M_i)=0$ because $\theta' \in H$.
Thus, $\theta'(Y)=\theta'(N)\ge 0$.
Consequently, $X \in \calW_{\theta'}$.

(3)
We get the ``if'' part straightforwardly.
Conversely, assume $H=H'$.
For any $X \in \supp_\theta M$, we get $X \in \calW_{\theta'}$ by (2)
and $\theta' \in \Theta_M \cap H' = \Theta_M \cap H$.
Take a nonzero quotient module $Y$ of $X$ which is simple in $\calW_{\theta'}$.
Then, by (2) again, $Y \in \calW_\theta$.
Since $X$ is a simple object in $\calW_\theta$ and $Y$ is a quotient module of $X$,
we have $Y=X$.
Therefore, $X$ is a simple object of $\calW_{\theta'}$.
This property holds for all $X \in \supp_\theta M$,
so a composition series of $M$ in $\calW_\theta$ is 
a composition series of $M$ in $\calW_{\theta'}$.
Thus, $\supp_\theta M=\supp_{\theta'} M$.
\end{proof}

%

The properties above yield that $\dim \Theta_M=n-1$
if and only if there exists $\theta \in \Theta_M$ such that
any $X \in \supp_\theta M$ satisfies $[X] \in \Q[M]$.
We remark that these conditions are not equivalent to that 
$M$ admits $\theta \in \Theta_M$ 
such that $M$ is a simple object in $\calW_\theta$,
because $\Theta_M=\Theta_{M \oplus M}$ holds for any $M \in \mod A \setminus \{ 0 \}$.

Moreover, the following property tells us that 
the dimension of every maximal wall with respect to inclusion is $n-1$
and that such a wall is always realized by a brick.

\begin{Prop}\label{Prop_max_wall_brick}
Let $M \in \mod A \setminus \{ 0 \}$.
Then, there exists $S \in \brick A$ such that $\Theta_S \supset \Theta_M$
and that $\dim \Theta_S=n-1$.
\end{Prop}

\begin{proof}
Take $\theta \in \Theta_M$ such that $\theta$ does not belong to 
any proper subface of $\Theta_M$, and set $H:=\Ker \ang{?}{\supp_\theta M}$.
By Lemma \ref{Lem_face} (1), $\Theta_M \cap H$ is the smallest face 
containing $\theta$, but it must be $\Theta_M$ itself.
Thus, we get $\Theta_M \cap H = \Theta_M$.
This and Lemma \ref{Lem_face} (2) imply that every $\theta' \in \Theta_M$ satisfies
$\supp_\theta M \subset \calW_{\theta'}$.
Now, we take $S \in \supp_\theta M$, 
then $S \in \calW_{\theta'}$ holds for all $\theta' \in \Theta_M$,
and this means $\Theta_S \supset \Theta_M$.
\end{proof}

Next, we show that the wall $\Theta_M$ can be given in terms of the walls $\Theta_{M'}$
for other modules $M'$ such that $\dim_K M'<\dim_K M$.
This will be crucial in Section \ref{Sec_path}
to determine the wall-chamber structure for path algebras inductively.
For $M',M'' \in \mod A$, set $M'*M''$ as the collection of modules $M \in \mod A$
admitting a short exact sequence $0 \to M' \to M \to M'' \to 0$.

\begin{Prop}\label{Prop_wall_ext}
Suppose that $M \in \mod A$ and that $\# \supp M \ge 3$.
Then, $\Theta_M$ is the smallest polyhedral cone of 
$K_0(\proj A)_\R$ containing
\begin{align*}
\bigcup_{\begin{smallmatrix}
M',M'' \in \mod A \setminus \{0\}\\
M \in M'*M''
\end{smallmatrix}}(\Theta_{M'} \cap \Theta_{M''}).
\end{align*}
\end{Prop}

To prove this, we need the following geometrical property.

\begin{Lem}\label{Lem_boundary_gen}
Let $C \subset \R^m$ be a strongly convex polyhedral cone,
and assume $\dim C \ge 2$.
We define $\partial C$ as the boundary of $C$ in $\R^m$.
Then, $C$ coincides with the smallest polyhedral cone in $\R^m$ containing 
$\partial C$. 
\end{Lem}

\begin{proof}
First, consider the canonical sphere
\begin{align*}
\Sigma:=\{(a_1,\ldots,a_m) \in \R^m \mid a_1^2+\cdots+a_m^2=1 \}.
\end{align*}
By the assumption $\dim C \ge 2$, we have $m \ge 2$,
so $\Sigma$ is connected.
On the other hand, $(C \cap \Sigma) \cup (-C \cap \Sigma)$ is a disjoint union
of two proper closed subsets of $\Sigma$,
since $C$ is strongly convex.
Therefore, the disjoint union $(C \cap \Sigma) \cup (-C \cap \Sigma)$
cannot be equal to the connected space $\Sigma$.
Thus, we can take $v_0 \in \Sigma$ satisfying $v_0 \notin C \cup (-C)$.

Now, let $v$ be in the interior of $C$ in $\R^m$. 
It is enough to find $u_1,u_2 \in \partial C$ such that $v=u_1+u_2$.
We define a subset $\Gamma:=\{ r \in \R \mid v+rv_0 \in C \}$ of $\R$.
Clearly, $\Gamma$ is convex in $\R$,
and $0 \in \R$ is in the interior of $\Gamma$ in $\R$,
since $C$ is a neighborhood of $v$ in $\R^m$.
Also, $\Gamma$ is bounded from above; 
otherwise, $(1/r)v+v_0 \in C$ holds for all $r>0$, 
so we get $v_0 \in C$ because $C$ is closed in $\R^m$, 
but it contradicts the choice of $v_0$.
Similarly, $\Gamma$ is bounded from below.
Thus, $\Gamma$ is actually a bounded closed interval $[r_1,r_2]$ with $r_1<0<r_2$,
and $v+r_1v_0, v+r_2v_0 \in \partial C$.
Then, the equation
\begin{align*}
v=\frac{r_2}{r_2-r_1}(v+r_1v_0) + \frac{-r_1}{r_2-r_1}(v+r_2v_0)
\end{align*}
implies that $v$ is a sum of two points in $\partial C$.
\end{proof}

We apply Lemma \ref{Lem_boundary_gen} by setting $m:=n-1$ to prove Proposition \ref{Prop_wall_ext}.

\begin{proof}[Proof of Proposition \ref{Prop_wall_ext}]
Set $H_1:=\bigoplus_{i \in \supp M} \R[S_i]$ and 
$H_2:=\bigoplus_{j \notin \supp M} \R[S_j]$.
Then, $\Theta_M=(\Theta_M \cap H_1) \oplus H_2$ and
$\Theta_{M'} \cap \Theta_{M''}=(\Theta_{M'} \cap \Theta_{M''} \cap H_1) \oplus H_2$
hold for $M',M'' \in \mod A \setminus \{0\}$ satisfying $M \in M'*M''$.
Thus, we may assume that $M$ is sincere.
By Lemma \ref{Lem_strongly_convex} (3), $\Theta_M$ is strongly convex.
Then, $n \ge 3$ follows by assumption.
We write $C$ for the smallest polyhedral cone of 
$K_0(\proj A)_\R$ containing
\begin{align*}
C':=\bigcup_{\begin{smallmatrix}
M',M'' \in \mod A \setminus \{0\} \\
M \in M'*M''
\end{smallmatrix}}
(\Theta_{M'} \cap \Theta_{M''}).
\end{align*}

We first show $C \subset \Theta_M$.
Since $\Theta_M$ is a polyhedral cone, it suffices to check $C' \subset \Theta_M$.
Assume that $\theta \in C'$, then we can take $M', M'' \in \mod A \setminus \{0\}$
satisfying $\theta \in \Theta_{M'} \cap \Theta_{M''}$ and $M \in M'*M''$.
By definition, $M',M'' \in \calW_\theta$, 
which implies $M \in M'*M'' \subset \calW_\theta$.
Thus, $\theta \in \Theta_M$, and we have $C' \subset \Theta_M$; hence $C \subset \Theta_M$.

Therefore, it remains to show the converse $\Theta_M \subset C$.


We set $\partial \Theta_M$ as the boundary of $\Theta_M$ 
in the hyperplane $\Ker \ang{?}{M}$.

We claim that $\Theta_M$ is the smallest polyhedral cone containing $\partial \Theta_M$.
If $\dim \Theta_M \le n-2$, it is clear.
If $\dim \Theta_M = n-1$,
then $\dim \Theta_M=n-1 \ge 2$ and $\Theta_M$ is strongly convex,
so Lemma \ref{Lem_boundary_gen} implies that $\Theta_M$ gives the claim.

Thus, it suffices to check $\partial \Theta_M \subset C$.
Let $\theta \in \partial \Theta_M$.
Then, $\theta$ belongs to some face $F \subset \Theta_M$ with $\dim F \le n-2$,
so $M$ is not a simple object in $\calW_\theta$ by Lemma \ref{Lem_face}.
As in the previous case, we can show $\partial \Theta_M \subset C$,
and $\Theta_M \subset C$.

We have obtained the assertion $\Theta_M=C$ as desired.
\end{proof}

\subsection{TF equivalence}

To investigate the walls $\Theta_M$ more, we will use numerical torsion pairs,
so we here shortly recall the definition of \textit{torsion pairs}.
Let $\calT,\calF$ be two full subcategories of $\mod A$,
then the pair $(\calT,\calF)$ is called \textit{a torsion pair} in $\mod A$
if $\Hom_A(\calT,\calF)=0$ holds and every $M \in \mod A$ admits a short exact sequence
$0 \to M' \to M \to M'' \to 0$ with $M' \in \calT$ and $M'' \in \calF$.
A full subcategory $\calT \subset \mod A$ is called a \textit{torsion class}
if $\calT$ admits $\calF \subset \mod A$ such that $(\calT,\calF)$ 
is a torsion pair in $\mod A$,
and this condition is equivalent to that $\calT$ is closed under taking
extensions and quotient modules.
Dually, \textit{torsion-free classes} in $\mod A$ are defined,
and they are precisely the full subcategories closed under taking
extensions and submodules.

Now, we associate two torsion classes and two torsion-free classes
to each $\theta \in K_0(\proj A)_\R$ 
as in Baumann--Kamnitzer--Tingley \cite{BKT}.
See also \cite[Lemma 6.6]{Bridgeland}.

\begin{Def}\cite[Subsection 3.1]{BKT}
Let $\theta \in K_0(\proj A)_\R$.
Then, we define \textit{numerical torsion classes} $\ovcalT_\theta$ and $\calT_\theta$ as
\begin{align*}
\ovcalT_\theta &:= \{M \in \mod A \mid 
\text{for any quotient module $X$ of $M$, $\theta(X) \ge 0$} \}, \\
\calT_\theta &:= \{M \in \mod A \mid 
\text{for any quotient module $X \ne 0$ of $M$, $\theta(X) > 0$} \}.
\end{align*}
Dually, we define \textit{numerical torsion-free classes} 
$\ovcalF_\theta$ and $\calF_\theta$ as
\begin{align*}
\ovcalF_\theta &:= \{M \in \mod A \mid 
\text{for any submodule $X$ of $M$, $\theta(X) \le 0$} \}, \\
\calF_\theta &:= \{M \in \mod A \mid 
\text{for any submodule $X \ne 0$ of $M$, $\theta(X) < 0$} \}.
\end{align*}
\end{Def}

Clearly, 
$\ovcalT_\theta \supset \calT_\theta$ 
and $\ovcalF_\theta \supset \calF_\theta$ hold, 
and their ``differences'' are expressed by the $\theta$-semistable subcategory
$\calW_\theta=\ovcalT_\theta \cap \ovcalF_\theta$.
Thus, the three conditions 
$\ovcalT_\theta=\calT_\theta$, $\ovcalF_\theta=\calF_\theta$ and 
$\calW_\theta=\{0\}$ are all equivalent.
The numerical torsion(-free) classes form torsion pairs in $\mod A$ as follows.

\begin{Prop}\cite[Proposition 3.1]{BKT}
For $\theta \in K_0(\proj A)_\R$,
the pairs $(\ovcalT_\theta,\calF_\theta)$ and $(\calT_\theta,\ovcalF_\theta)$ 
are torsion pairs in $\mod A$.
\end{Prop}

Therefore, $\ovcalT_\theta$ and $\calF_\theta$ determine each other,
and so do $\calT_\theta$ and $\ovcalF_\theta$.
By using numerical torsion(-free) classes,
we introduce an equivalence relation on $K_0(\proj A)_\R$ as follows.

\begin{Def}
Let $\theta$ and $\theta'$ be elements in $K_0(\proj A)_\R$.
We say that $\theta$ and $\theta'$ are \textit{TF equivalent} if both
$\ovcalT_\theta=\ovcalT_{\theta'}$ and $\ovcalF_\theta=\ovcalF_{\theta'}$ hold.
We define $[\theta] \subset K_0(\proj A)_\R$ as the TF equivalence class of $\theta$.
\end{Def}

Now, we give an example on TF equivalence classes.

\begin{Ex}
Let $A$ be the path algebra $K(1 \to 2)$.
Since the Auslander--Reiten quiver of $\mod A$ is 
\begin{align*}
\begin{xy}
( 0,-4) *+{S_2} = "1",
( 8, 4) *+{P_1} = "2",
(16,-4) *+{S_1} = "3",
\ar "1";"2"
\ar "2";"3"
\end{xy},
\end{align*}
we can express an additive full subcategory $\calC$ of $\mod A$ 
by writing $\bullet$ or $\circ$
instead of each $*$ in the diagram
$\begin{xy}
( 0,-1) *{*} = "1",
( 2, 1) *{*} = "2",
( 4,-1) *{*} = "3",
\end{xy}$:
each $*$ corresponds to the module in the same place in the Auslander--Reiten quiver,
and $\bullet$ means that the module belongs to $\calC$ and $\circ$ means not.
For example,
$\begin{xy}
( 0,-1) *{\circ} = "1",
( 2, 1) *{\bullet} = "2",
( 4,-1) *{\bullet} = "3",
\end{xy}$ 
denotes $\add\{P_1,S_1\}$.

Under this notation, 
$\ovcalT_\theta$ and $\ovcalF_\theta$ for $\theta \in K_0(\proj A)_\R$
are as in the following pictures, respectively.
Each domain contains a line or a point in its boundary 
if it is described by a solid line or a black point,
and does not contain if it is denoted by a dotted line or a white point.

\begin{align*}
\includegraphics[width=5.7cm]{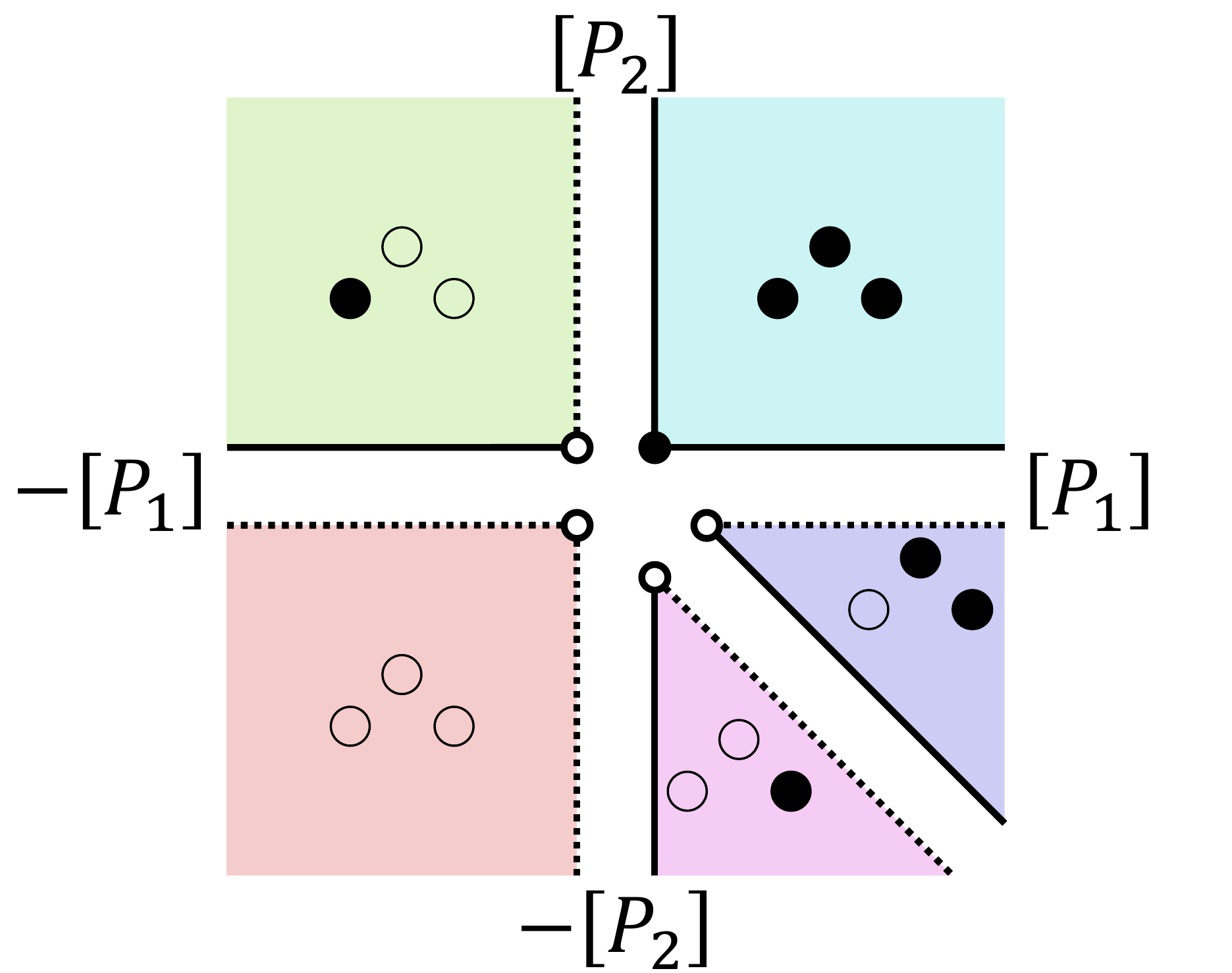} \qquad
\includegraphics[width=5.7cm]{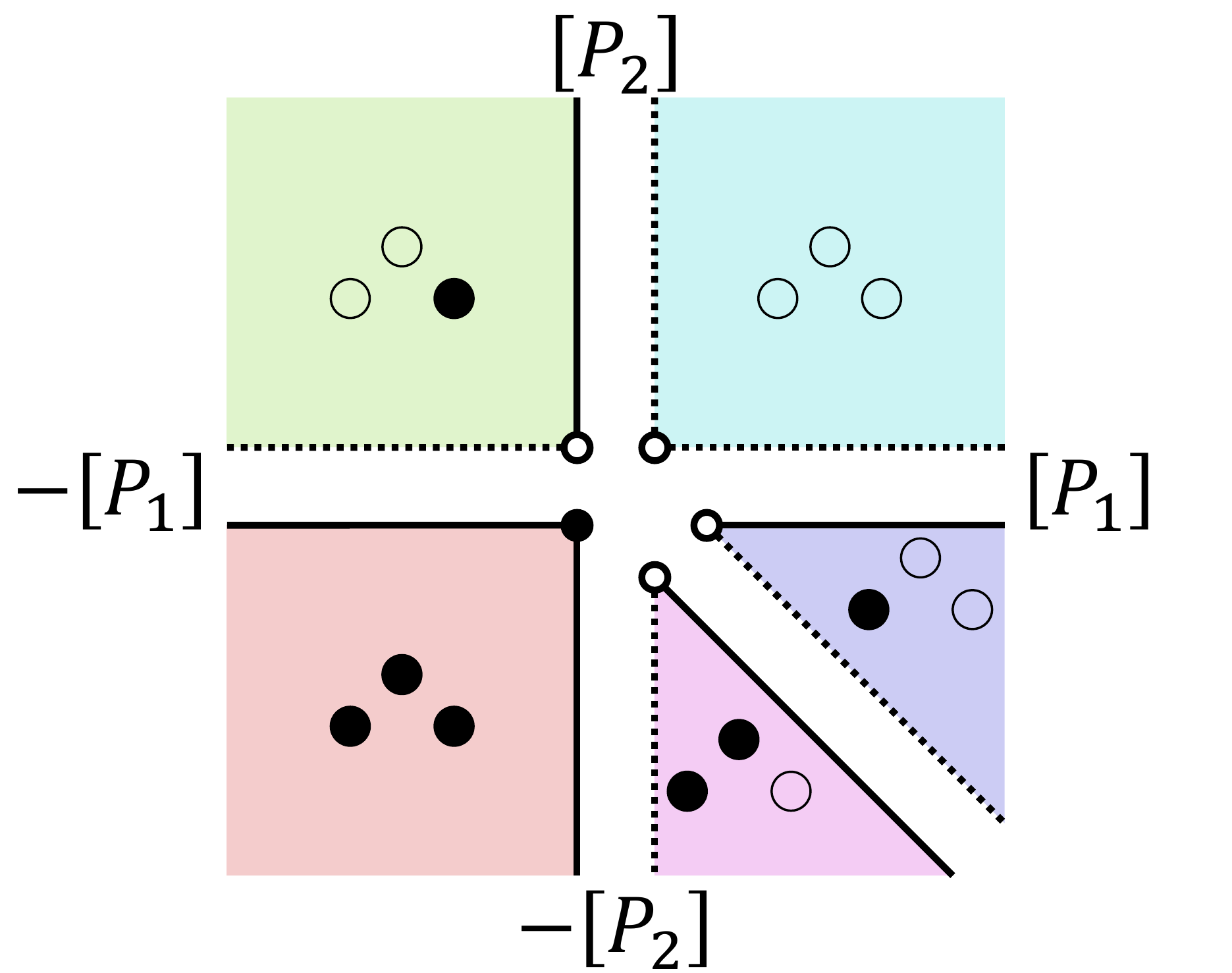}
\end{align*}
Therefore, $K_0(\proj A)_\R$ is divided to eleven TF equivalence classes,
which are the origin, the five half-lines without the origin, 
and the five colored open domains in the following picture:
\begin{align*}
\includegraphics[width=5.7cm]{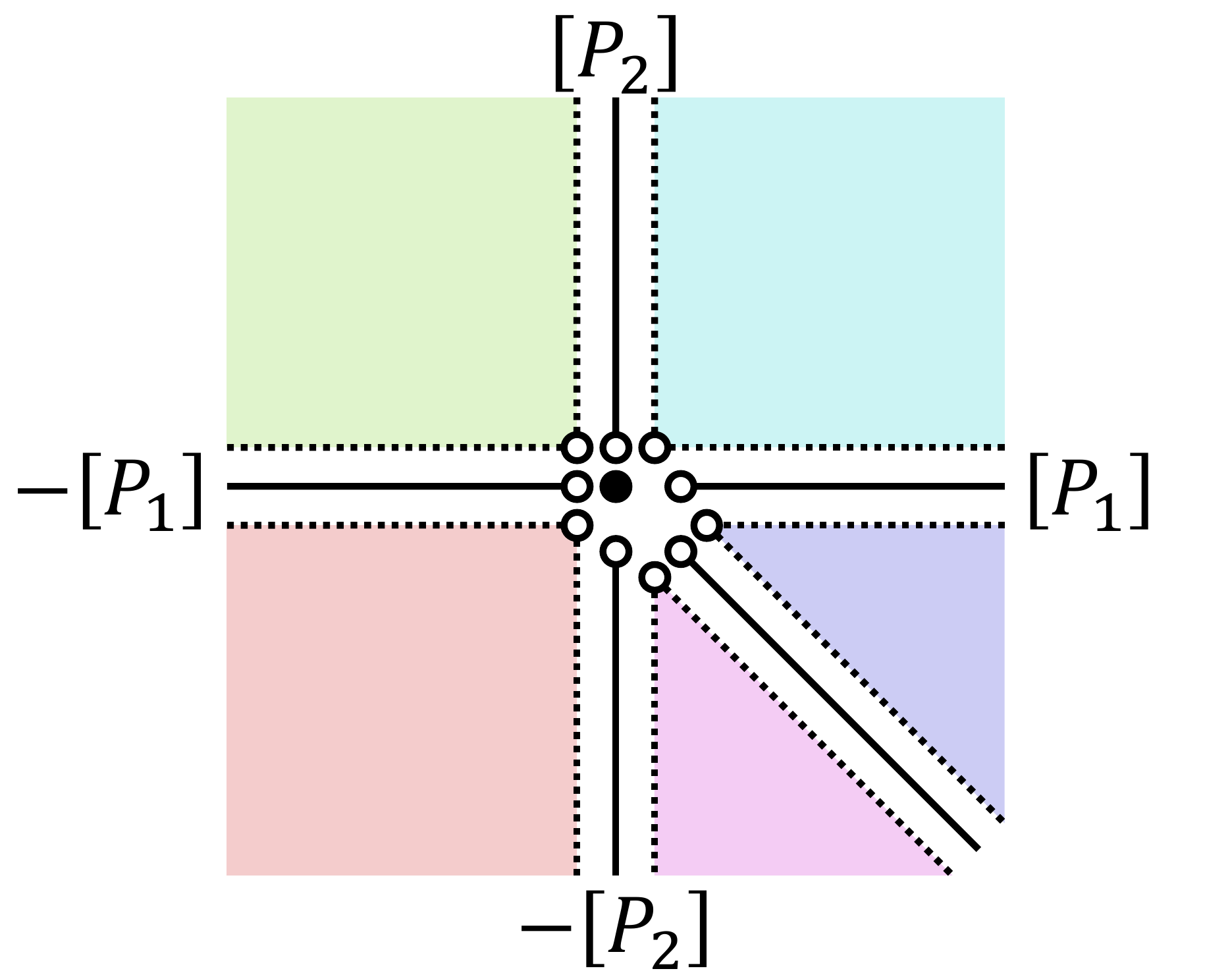}
\end{align*}

\end{Ex}

The following property is easily deduced, but important.
We write 
\begin{align*}
[\theta,\theta'] := \{ (1-r)\theta+r\theta' \mid r \in [0,1] \}
\end{align*}
for the line segment between $\theta,\theta' \in K_0(\proj A)_\R$.

\begin{Lem}\label{Lem_TF_convex}
In $K_0(\proj A)_\R$, any TF equivalence class is convex.
\end{Lem}

\begin{proof}
Assume that $\theta$ and $\theta'$ in $K_0(\proj A)_\R$ are TF equivalent 
and that $\theta'' \in [\theta,\theta']$.
By definition, 
$\ovcalT_{\theta''}$ contains $\ovcalT_\theta \cap \ovcalT_{\theta'}$,
which is equal to $\ovcalT_\theta$ by assumption.
Similarly, $\calF_{\theta''}$ contains $\calF_\theta \cap \calF_{\theta'}=\calF_\theta$.
Thus, the torsion pair $(\ovcalT_{\theta''}, \calF_{\theta''})$
must coincide with $(\ovcalT_\theta,\calF_\theta)$.
We can also prove that 
$(\calT_{\theta''}, \ovcalF_{\theta''})=(\calT_\theta,\ovcalF_\theta)$
in the same way.
Therefore, $\theta''$ is also TF equivalent to $\theta$.
\end{proof}

Each TF equivalence class is not a closed subset in general,
but the closure enjoys the following nice property.

\begin{Lem}
Let $\theta,\theta' \in K_0(\proj A)_\R$.
Then, $\theta'$ belongs to the closure $\overline{[\theta]}$
if and only if 
$\ovcalF_\theta \subset \ovcalF_{\theta'}$ and $\ovcalT_\theta \subset \ovcalT_{\theta'}$.
\end{Lem}

\begin{proof}
The ``only if'' part follows from the definition.

Conversely, assume
$\ovcalF_\theta \subset \ovcalF_{\theta'}$ and $\ovcalT_\theta \subset \ovcalT_{\theta'}$.
In this case, we can show that any $\theta'' \in [\theta,\theta'] \setminus \{\theta'\}$
satisfies $\ovcalT_{\theta''} \supset \ovcalT_\theta \cap \ovcalT_{\theta'}
=\ovcalT_\theta$ and
$\calF_{\theta''} \supset \calF_\theta \cap \ovcalF_{\theta'} \supset 
\calF_\theta \cap \ovcalF_{\theta} = \calF_\theta$.
Thus, we have $(\ovcalT_{\theta''}, \calF_{\theta''})=(\ovcalT_\theta,\calF_\theta)$.
Similarly, $(\calT_{\theta''}, \ovcalF_{\theta''})=(\calT_\theta,\ovcalF_\theta)$ holds.
Therefore, every $\theta'' \in [\theta,\theta'] \setminus \{\theta'\}$ 
belongs to the TF equivalence class $[\theta]$,
so $\theta'$ is in the closure $\overline{[\theta]}$.
\end{proof}

Moreover, we are able to characterize TF equivalence
in terms of the walls $\Theta_M$ as follows.

\begin{Thm}\label{Thm_W_constant}
Let $\theta,\theta' \in K_0(\proj A)_\R$ be distinct elements.
Then, the following conditions are equivalent.
\begin{itemize}
\item[(a)]
The elements $\theta$ and $\theta'$ are TF equivalent.
\item[(b)]
Any $\theta'' \in [\theta,\theta']$ is TF equivalent to $\theta$.
\item[(c)]
For any $\theta'' \in [\theta,\theta']$, 
the $\theta''$-semistable subcategory $\calW_{\theta''}$ is constant.
\item[(d)]
For any module $M$, we have $[\theta,\theta'] \cap \Theta_M = \emptyset$ or 
$[\theta,\theta'] \subset \Theta_M$.
\item[(e)]
There does not exist $S \in \brick A$ such that 
$[\theta,\theta'] \cap \Theta_S$ has exactly one element.
\end{itemize}
\end{Thm}

To prove this, we also prepare a fact on coincidence of torsion pairs.

\begin{Lem}\label{Lem_TF_coincide}
Let $(\calT,\calF)$ and $(\calT',\calF')$ be torsion pairs in $\mod A$.
Then, $(\calT,\calF)$ and $(\calT',\calF')$ coincide 
if and only if $\calT \cap \calF'=\calT' \cap \calF=\{ 0 \}$.
\end{Lem}

\begin{proof}
If $(\calT,\calF)=(\calT',\calF')$,
then we clearly have $\calT \cap \calF'=\calT' \cap \calF=\{ 0 \}$.

Conversely, we assume $\calT \cap \calF'= \{0\}$ and $\calT' \cap \calF=\{ 0 \}$.
Then, $\Hom_A(\calT,\calF')=0$ and $\Hom_A(\calT',\calF)=0$ hold,
and they imply $\calF' \subset \calF$ and $\calF \subset \calF'$, respectively.
Therefore, we get $\calF=\calF'$ and $\calT=\calT'$.
\end{proof}

The following criterion is obtained by simply
applying Lemma \ref{Lem_TF_coincide} to the numerical torsion pairs
$(\ovcalT_\theta,\calF_\theta)$ and $(\calT_\theta,\ovcalF_\theta)$.

\begin{Lem}\label{Lem_TF_coincide_num}
Let $\theta, \theta' \in K_0(\proj A)_\R$.
Then the following assertions hold.
\begin{itemize}
\item[(a)]
The torsion classes $\ovcalT_\theta$ and $\ovcalT_{\theta'}$ coincide
if and only if 
$\ovcalT_\theta \cap \calF_{\theta'}=\ovcalT_{\theta'} \cap \calF_\theta = \{ 0 \}$.
\item[(b)]
The torsion-free classes $\ovcalF_\theta$ and $\ovcalF_{\theta'}$ coincide
if and only if 
$\calT_\theta \cap \ovcalF_{\theta'}=\calT_{\theta'} \cap \ovcalF_\theta = \{ 0 \}$.
\end{itemize}
\end{Lem}

Now, we can prove Theorem \ref{Thm_W_constant}.

\begin{proof}[Proof of Theorem \ref{Thm_W_constant}]
$\text{(a)} \Rightarrow \text{(b)}$ follows from Lemma \ref{Lem_TF_convex}.

$\text{(b)} \Rightarrow \text{(c)}$ and $\text{(c)} \Rightarrow \text{(d)}$ 
are clear by definition.

$\text{(d)} \Rightarrow \text{(e)}$ is also obvious, since $\theta \ne \theta'$.

$\text{(e)} \Rightarrow \text{(a)}$:
We assume that $\theta$ and $\theta'$ are not TF equivalent and 
will find some $S \in \brick A$ such that 
$[\theta,\theta'] \cap \Theta_S$ has exactly one element.

Then, since $\theta$ and $\theta'$ are not TF equivalent, 
$\ovcalT_\theta \ne \ovcalT_{\theta'}$ 
or $\ovcalF_\theta \ne \ovcalF_{\theta'}$ holds.
We only consider the case $\ovcalT_\theta \ne \ovcalT_{\theta'}$,
because a similar proof works in the other case.

In this case, we have 
$\ovcalT_\theta \cap \calF_{\theta'} \ne \{0\}$ or 
$\ovcalT_{\theta'} \cap \calF_\theta \ne \{0\}$
from Lemma \ref{Lem_TF_coincide_num}.
By exchanging $\theta$ and $\theta'$,
we may assume $\ovcalT_\theta \cap \calF_{\theta'} \ne \{0\}$.
We can take a nonzero module $S \in \ovcalT_\theta \cap \calF_{\theta'}$ such that
$\dim_K S \le \dim_K M$ holds for all nonzero modules 
$M \in \ovcalT_\theta \cap \calF_{\theta'}$.
Then, $S$ is a brick by \cite[Lemma 3.8]{DIRRT}.

Since $S \in \ovcalT_\theta \cap \calF_{\theta'}$,
we have $\theta(S) \ge 0$ and $\theta'(S) < 0$.
Thus, there uniquely exists $\theta'' \in [\theta,\theta']$
such that $\theta''(S)=0$, 
and it suffices to show $\theta'' \in \Theta_S$.

By minimality,
any proper nonzero quotient module of $S$ must belong to $\ovcalT_{\theta'}$.
We have $S \in \ovcalT_\theta \cap \ovcalT_{\theta'}$,
and this clearly implies $S \in \ovcalT_{\theta''}$.
Therefore, $S \in \calW_{\theta''}$, 
which means that $\theta'' \in \Theta_S$.
\end{proof}

In general, there may exist infinitely many TF equivalence classes.

\begin{Ex}
Let $A$ be the path algebra of the $m$-Kronecker quiver:
\begin{align*}
\begin{xy}
( 0, 0)*+{1}="1",
(20, 0)*+{2}="2",
(10, 0)*+{\vdots},
\ar@<4mm> "1";"2"
\ar@<2mm> "1";"2"
\ar@<-4mm> "1";"2"
\end{xy}
\quad \text{($m$ arrows)}.
\end{align*}
In Example \ref{Ex_Kronecker} later, 
we give the wall-chamber structure of $K_0(\proj A)_\R$,
and this and Theorem \ref{Thm_W_constant} 
tell us the cardinality of the set of TF equivalence classes as follows:
\begin{itemize}
\item
if $m=0,1$, then only finitely many TF equivalence classes exist;
\item
if $m=2$, then the set of TF equivalence classes is infinite and countable;
\item
if $m \ge 3$, then there exist uncountably many TF equivalence classes.
\end{itemize}
\end{Ex}

As an application of Theorem \ref{Thm_W_constant},
We consider the chambers in $K_0(\proj A)_\R$.
Here, a \textit{chamber} means a connected component of the open subset 
$K_0(\proj A)_\R \setminus \overline{\Wall}$,
where $\Wall:=\bigcup_{M \in \mod A \setminus \{0\}} \Theta_M$ is 
the union of all walls $\Theta_M$.
We set $\Chamber(A)$ as the set of all chambers in $K_0(\proj A)_\R$.
By Theorem \ref{Thm_W_constant}, any chamber $C$ is a subset of
some TF equivalence class of dimension $n$ (that is,
the interior of the TF equivalence class is nonempty),
because any two elements $\theta,\theta' \in C$ have a polyline in $C$ connecting them.
This correspondence gives a map $\Chamber(A) \to \TF_n(A)$,
where $\TF_n(A)$ denotes the set of
all TF equivalence classes of dimension $n$ in $K_0(\proj A)_\R$.
Actually, we can prove that this map is a bijection
as follows.

\begin{Cor}\label{Cor_TF_chamber}
We have mutually inverse bijections
\begin{align*}
\Chamber(A) &\leftrightarrow \TF_n(A) \\
C &\mapsto \textup{(the TF equivalence class containing $C$)}, \\
E^\circ &\mapsfrom E,
\end{align*}
where $E^\circ$ means the interior of $E$ in $K_0(\proj A)_\R$.
\end{Cor}

\begin{proof}
First, we must show that the second map $E \mapsto E^\circ$ is well-defined.
Let $E \in \TF_n(A)$.

Then, $E$ cannot be contained in any wall $\Theta_M$ for $M \in \mod A \setminus \{0\}$,
so Theorem \ref{Thm_W_constant} implies that 
$\calW_\theta=\{0\}$ holds for all $\theta \in E$.
Thus, $E \subset K_0(\proj A)_\R \setminus \Wall$.

Since $E^\circ$ is an open subset, 
we get $E^\circ \subset K_0(\proj A)_\R \setminus \overline{\Wall}$.
Since $E$ is convex in $K_0(\proj A)_\R$, so is $E^\circ$.
Thus, $E^\circ$ is contained in a connected component $C$ of 
$K_0(\proj A)_\R \setminus \overline{\Wall}$.
By definition, $C \in \Chamber(A)$.

The chamber $C$ is contained in some TF equivalence class,
which must be $E$, since $\emptyset \ne E^\circ \subset C$.
Thus, we have $E^\circ \subset C \subset E$.
Because $C$ is an open set,
we get $E^\circ=C \in \Chamber(A)$.

Now, we can easily check that the two maps are mutually inverse.
\end{proof}

\subsection{Numerical torsion(-free) classes are widely generated}

We end this section by considering the relationship between
numerical torsion pairs and widely generated torsion(-free) classes.
Let us write $\wide A$ for the set of wide subcategories
and $\tors A$ for the set of torsion classes in $\mod A$.
It is easy to see that every $\calW \in \wide A$ has 
the smallest torsion class $\sfT(\calW)$ containing $\calW$.
As in \cite{AP}, a torsion class $\calT \in \tors A$ is said to be 
\textit{widely generated} if there exists a wide subcategory $\calW \in \wide A$ 
such that $\calT=\sfT(\calW)$.
Similarly, the notion of widely generated torsion-free classes is defined.

Widely generated torsion classes have been characterized in several ways
\cite{AP,BCZ,MS}. 
In our context, the following one is crucial.

\begin{Prop}\label{Prop_MS}\cite[Proposition 3.3]{MS}
Let $\calT$ be a torsion class in $\mod A$.
Then, $\calT$ is a widely generated torsion class if and only if
$\calT=\sfT(\alpha(\calT))$,
where 
\begin{align*}
\alpha(\calT):=\{ M \in \calT \mid \text{
for any $f \colon N \to M$ with $N \in \calT$, $\Ker f \in \calT$} \} \in \wide A.
\end{align*}
\end{Prop}

We aim to show the following property on the subset
\begin{align*}
K_0(\proj A)_\Q &:= \{ a_1[P_1]+\cdots+a_n[P_n] \mid a_1,\ldots,a_n \in \Q \} \subset 
K_0(\proj A)_\R.
\end{align*}

\begin{Thm}\label{Thm_widely_gen}
Let $\theta \in K_0(\proj A)_\Q$.
Then, $\ovcalT_\theta$ is a widely generated torsion class, and
$\ovcalF_\theta$ is a widely generated torsion-free class. 
\end{Thm}

\begin{Rem}
For any $\theta \in K_0(\proj A)_\R$,
\begin{align*}
\calW_\theta=\{ X \in \alpha(\ovcalT_\theta) \mid \theta(X)=0 \}
\end{align*}
holds, but in general, $\calW_\theta$ does not coincide 
with $\alpha(\ovcalT_\theta)$.
\end{Rem}

By Proposition \ref{Prop_MS}, 
if $\calT$ is a torsion class which is not widely generated,
then $\sfT(\alpha(\calT)) \subsetneq \sfT$,
so we get $\calT \cap \alpha(\calT)^\perp \ne \{0\}$.
We need the following technical property.

\begin{Lem}\label{Lem_infin_seq}
Assume that $(\calT,\calF)$ is a torsion pair in $\mod A$ and that
$\calT$ is not a widely generated torsion class.
Then, there exists an infinite sequence
\begin{align*}
\cdots \xrightarrow{f_2} M_2 \xrightarrow{f_1} M_1 \xrightarrow{f_0} M_0=M
\end{align*}
of homomorphisms such that 
$M_i \in \calT$ and $0 \ne \Ker f_i \in \calF$ for each $i \ge 0$.
\end{Lem}

\begin{proof}
Set $\calC:=\calT \cap \alpha(\calT)^\perp$,
which is not $\{0\}$ by the observation above.
We can take a nonzero module $M \in \calC \setminus \{0\}$.

Since $M \notin \alpha(\calT)$, there exists some $M_1 \in \calT$ and 
a homomorphism $f_0 \colon M_1 \to M_0$ such that $\Ker f_0 \notin \calT$.
We may choose $M_1$ and $f_0$ so that $\dim M_1$ is the smallest possible.

Under this condition, we show that $\Ker f_0$ must belong to $\calF$.
We can take an exact sequence 
$0 \to L' \to \Ker f_0 \xrightarrow{\pi} L \to 0$ with $L' \in \calT$ and $L \in \calF$.
Since $\Ker f_0 \notin \calT$, we get $L \ne 0$.
By considering the push-out, we have the following commutative diagram
with the rows exact:
\begin{align*}
\begin{xy}
( 0, 8)*+{0}="1",
(24, 8)*+{\Ker f_0}="2",
(48, 8)*+{M_1}="3",
(72, 8)*+{\Im f_0}="4",
(96, 8)*+{0}="5",
( 0,-8)*+{0}="6",
(24,-8)*+{L}="7",
(48,-8)*+{M'_1}="8",
(72,-8)*+{\Im f_0}="9",
(96,-8)*+{0}="10"
\ar "1";"2"
\ar "2";"3"
\ar "3";"4"
\ar "4";"5"
\ar "6";"7"
\ar "7";"8"
\ar^{f'_0} "8";"9"
\ar "9";"10"
\ar_{\pi} "2";"7"
\ar_{\pi'} "3";"8"
\ar@{=} "4";"9"
\end{xy}.
\end{align*}
The morphism $f'_0 \colon M'_1 \to \Im f_0$ $(\subset M_0)$ satisfies 
$\Ker f'_0 \cong L \notin \calT$, since $L \ne 0$ and $L \in \calF$.
By the choice of $M_1$, we obtain $\dim M'_1 \ge \dim M_1$.
On the other hand, 
the homomorphism $\pi' \colon M_1 \to M'_1$ above is surjective.
Thus, $M'_1 \cong M_1$, and $\pi \colon \Ker f_0 \cong L \in \calF$ as desired.

Next, we prove that $M_1$ also belongs to $\calC$.
Let $W \in \alpha(\calT)$ and $g \colon W \to M_1$.
It suffices to show $g=0$.
Since $M_0 \in \alpha(\calT)^\perp$, the composite $f_0 g \colon W \to M_0$ must be zero.
Thus, $f_0$ induces $f''_0 \colon M_1/\Im g \to M_0$,
and there exists a short exact sequence
$0 \to \Im g \to \Ker f_0 \to \Ker f''_0 \to 0$.
Since $\alpha(\calT) \subset \calT$, we have $\Im g \in \calT$.
On the other hand, since $\Ker f_0 \notin \calT$,
so $\Ker f''_0$ cannot belong to $\calT$.
By the choice of $M_1$ again, we have $\dim(M_1/\Im g) \ge \dim M_1$.
This implies that $\Im g=0$, or equivalently, $g=0$.

Now, we have proved that $M_1 \in \calC$,
so we can apply the process above again,
namely, we have some $f_1 \colon M_2 \to M_1$ 
with $\Ker f_1 \in \calF$ and $M_2 \in \calC$.
By repeating this, we obtain a family $(f_i \colon M_{i+1} \to M_i)_{i=0}^\infty$
of homomorphisms such that $M_i \in \calC$ and $0 \ne \Ker f_i \in \calF$.
\end{proof}

By using Lemma \ref{Lem_infin_seq}, Theorem \ref{Thm_widely_gen} can be proved.

\begin{proof}
We only prove the assertion for $\ovcalT_\theta$.
The other one can be similarly proved.

We assume that $\ovcalT_\theta$ is not widely generated, and deduce a contradiction.
From Lemma \ref{Lem_infin_seq}, take an infinite sequence
\begin{align*}
\cdots \xrightarrow{f_2} M_2 \xrightarrow{f_1} M_1 \xrightarrow{f_0} M_0=M
\end{align*}
of homomorphisms such that 
$M_i \in \ovcalT_\theta$ and $0 \ne \Ker f_i \in \calF_\theta$ for each $i \ge 0$.
Then, for every $i \ge 0$, we have the short exact sequence
\begin{align*}
0 \to \Ker f_i \to M_{i+1} \to M_i \to \Coker f_i \to 0,
\end{align*}
so $\theta(M_{i+1})=\theta(\Ker f_i)+\theta(M_i)-\theta(\Coker f_i)$.
Because $\Ker f_i \in \calF_\theta$, we obtain $\theta(\Ker f_i)<0$.
Moreover, since $M_i \in \ovcalT_\theta$, we get $\theta(\Coker f_i) \ge 0$.
Thus, $\theta(M_{i+1})<\theta(M_i)$ holds for every $i \ge 0$.
On the other hand, $M_i \in \ovcalT_\theta$ implies that $\theta(M_i) \ge 0$.

As a consequence, we have 
$0 \le \cdots < \theta(M_2) < \theta(M_1) < \theta(M_0)$,
but it contradicts to the assumption $\theta \in K_0(\proj A)_\Q$.
Therefore, the torsion class $\ovcalT_\theta$ must be widely generated.
\end{proof}

\section{The wall-chamber structures and the Koenig--Yang correspondences}

\subsection{Preparations on the Koenig--Yang correspondences}

In this section, we study the relationship between stability conditions
and the Koenig--Yang correspondences established in \cite{KY, BY}.
The Koenig--Yang correspondences are a collection of bijections
between many important notions in the perfect derived category $\sfK^\rmb(\proj A)$
and the bounded derived category $\sfD^\rmb(\mod A)$,
such as silting objects in $\sfK^\rmb(\proj A)$,
bounded t-structures with length heart in $\sfD^\rmb(\mod A)$,
and 2-term simple-minded collections in $\sfD^\rmb(\mod A)$. 

Before explaining the detail, we recall some notions here.

Let $\calC$ be a triangulated category.
We say that a triangulated subcategory $\calC' \subset \calC$ is 
\textit{thick} if $\calC'$ is closed under taking direct summands.

For every $U \in \sfK^\rmb(\proj A)$, 
we define the full subcategory $\add U \subset \sfK^\rmb(\proj A)$
as the additive closure of $U$,
that is, $U' \in \add U$ holds if and only if there exists some $s \in \Z_{\ge 0}$
such that $U'$ is isomorphic to a direct summand of $U^{\oplus s}$.

Every $U \in \sfK^\rmb(\proj A)$ admits a decomposition 
$U \cong \bigoplus_{i=1}^m U_i^{s_i}$ with 
all $U_i$ indecomposable, $s_i \ge 1$, and $U_i \not \cong U_j$ for any $i \ne j$.
We set $|U|:=m$, and say that $U$ is \textit{basic}
if $s_i=1$ for all $i$.

Now, we recall the definition of silting objects in $\sfK^\rmb(\proj A)$.

\begin{Def}\label{Def_silt}
We define the following notions.
\begin{itemize}
\item[(1)]
An object $U \in \sfK^\rmb(\proj A)$ is said to be \textit{presilting}
if $\Hom_{\sfK^\rmb(\proj A)}(U,U[{>}0])=0$.
\item[(2)]
An object $T \in \sfK^\rmb(\proj A)$ is said to be \textit{silting}
if $T$ is a presilting object and 
the smallest thick subcategory of $\sfK^\rmb(\proj A)$ containing $U$ is 
$\sfK^\rmb(\proj A)$ itself.
\end{itemize}
We write $\silt A$ for the set of isomorphism classes of basic silting objects in 
$\sfK^\rmb(\proj A)$.
\end{Def}

In this paper, we mainly focus on the 2-term versions of these notions.

\begin{Def}\label{Def_twosilt}
An object $U$ in $\sfK^\rmb(\proj A)$ is said to be
\textit{2-term} if $U$ is isomorphic 
to some complex $(P^{-1} \to P^0)$ whose terms except $-1$st and $0$th ones vanish.
We write $\twopresilt A$ (resp.~$\twosilt A$) for the set of isomorphism classes of 
basic 2-term presilting objects (resp.~2-term silting objects) in $\sfK^\rmb(\proj A)$.
\end{Def}

2-term presilting and silting objects satisfy many nice properties as below;
see also \cite[Corollary 5.1]{KY} for (4).

\begin{Prop}\label{Prop_presilt_silt}
Let $U \in \twopresilt A$.
Then the following assertions hold.
\begin{itemize}
\item[(1)]
\cite[Proposition 2.16]{Aihara}
There exists some $T \in \twosilt A$ such that $U \in \add T$.
\item[(2)]
\cite[Proposition 3.3]{AIR}
The condition $|U|=n$ is equivalent to $U \in \twosilt A$.
\item[(3)]
\cite[Corollary 3.8]{AIR}
If $|U|=n-1$, then $\# \{ T \in \twosilt A \mid U \in \add T \}=2$.
\item[(4)]
\cite[Theorem 2.27, Corollary 2.28]{AI}
The indecomposable direct summands of each $T \in \twosilt A$ gives a $\Z$-basis of 
$K_0(\proj A)$.
\end{itemize}
\end{Prop}

In the setting of Proposition \ref{Prop_presilt_silt} (3),
let $T,T'$ be the two distinct elements with $U \in \add T$ and $U \in \add T'$.
Then, $T'$ is called the \textit{mutation} of $T$
at the indecomposable direct summand $T/U$. 

Next, we recall the notion of t-structures, 
which is defined in the following way.

\begin{Def}\label{Def_tstr}
Let $\calU, \calV$ be full subcategories of $\sfD^\rmb(\mod A)$.
Then, the pair $(\calU, \calV)$ is called a \textit{t-structure} 
if the following conditions hold:
\begin{itemize}
\item[(a)]
$\calU[1] \subset \calU$ and $\calV[-1] \subset \calV$;
\item[(b)]
$\Hom_{\sfD^\rmb(\mod A)}(\calU,\calV[-1])=0$;
\item[(c)]
for any $X \in \sfD^\rmb(\mod A)$, there exists a triangle
$X' \to X \to X'' \to X'[1]$ with $X' \in \calU$ and $X'' \in \calV[-1]$.
\end{itemize}
A t-structure $(\calU,\calV)$ is said to be \textit{bounded} if
\begin{align*}
\bigcup_{i \in \Z} \calU[i] = \sfD^\rmb(\mod A) = \bigcup_{i \in \Z} \calV[i].
\end{align*}
For a t-structure $(\calU,\calV)$,
the intersection $\calU \cap \calV$ is called the \textit{heart}, 
which is an abelian category \cite{BBD}.
If the heart $\calU \cap \calV$ is an abelian length category, 
then the t-structure $(\calU,\calV)$ is said to be \textit{with length heart}.
We define $\tstr A$ as the set of bounded t-structures in $\sfD^\rmb(\mod A)$
with length heart.

Moreover, a t-structure $(\calU, \calV)$ is said to be \textit{intermediate}
if $\calD^{\le -1} \subset \calU \subset \calD^{\le 0}$, where
\begin{align*}
\calD^{\le k} := \{X \in \sfD^\rmb(\mod A) \mid \text{$H^i(X)=0$ for $i>k$}\}.
\end{align*}
An intermediate t-structure in $\sfD^\rmb(\mod A)$ is always bounded,
and we write $\inttstr A$ for the set of intermediate t-structures in $\sfD^\rmb(\mod A)$
with length heart.
\end{Def}

We also use 2-term simple-minded collections,
which are defined as follows.

\begin{Def}\label{Def_smc}
A set $\calX$ of isomorphism classes of objects in $\sfD^\rmb(\mod A)$ 
is called a \textit{simple-minded collection} in $\sfD^\rmb(\mod A)$ 
if the following conditions are satisfied:
\begin{itemize}
\item[(a)]
for any $X \in \calX$, 
the endomorphism ring $\End_{\sfD^\rmb(\mod A)}(X)$ is a division ring;
\item[(b)]
for any $X_1,X_2 \in \calX$ with $X_1 \ne X_2$, 
we have $\Hom_{\sfD^\rmb(\mod A)}(X_1,X_2)=0$;
\item[(c)]
for any $X_1,X_2 \in \calX$ and $k \in \Z_{<0}$, 
we have $\Hom_{\sfD^\rmb(\mod A)}(X_1,X_2[k])=0$;
\item[(d)]
the smallest thick subcategory of $\sfD^\rmb(\mod A)$ containing $\calX$ is 
$\sfD^\rmb(\mod A)$ itself.
\end{itemize} 
We write $\smc A$ for the set of simple-minded collections in $\sfD^\rmb(\mod A)$.

Moreover, a simple-minded collection $\calX$ in $\sfD^\rmb(\mod A)$ is called 
a \textit{2-term simple-minded collection} in $\sfD^\rmb(\mod A)$ if 
the $i$th cohomology $H^i(X)$ vanishes 
for any $X \in \calX$ and $i \in \Z\setminus \{-1,0\}$.
We write $\twosmc A$ for the set of 2-term simple-minded collections 
in $\sfD^\rmb(\mod A)$.
\end{Def}

Each simple-minded collection $\calX$ in $\sfD^\rmb(\mod A)$
has exactly $n$ elements \cite[Corollary 5.5]{KY},
which is equal to the number of 
indecomposable direct summands of a silting object in $\sfK^\rmb(\proj A)$.

The following bijections between $\silt A$, $\tstr A$ and $\smc A$ are
included in the Koenig--Yang correspondences.

\begin{Prop}\label{Prop_silt_smc}
The following assertions hold.
\begin{itemize}
\item[(1)]\cite[Theorem 6.1]{KY}
There exist the following bijections:
\begin{itemize}
\item[(a)]
$\silt A \to \tstr A$ sending 
$T \in \silt A$ to the t-structure $(T[{<}0]^\perp,T[{>}0]^\perp)$, where 
\begin{align*}
T[{<}0]^\perp &:= \{ U \in \sfK^\rmb(\proj A) \mid 
\text{$\Hom_{\sfD^\rmb(\mod A)}(T[k],U)=0$ holds for any $k < 0$} \}, \\
T[{>}0]^\perp &:= \{ U \in \sfK^\rmb(\proj A) \mid 
\text{$\Hom_{\sfD^\rmb(\mod A)}(T[k],U)=0$ holds for any $k > 0$} \};
\end{align*}
\item[(b)]
$\tstr A \to \smc A$ sending $(\calU, \calV) \in \tstr A$
to the set of isomorphism classes of simple objects of the heart $\calU \cap \calV$.
\end{itemize}
\item[(2)]\cite[Lemma 5.3]{KY}
Let $\calX \in \smc A$ correspond to $T \in \silt A$ under the bijections in (1).
Then, there exist families $(T_i)_{i=1}^n$ and $(X_i)_{i=1}^n$ satisfying
the following conditions:
\begin{itemize}
\item[(a)]
$T=\bigoplus_{i=1}^n T_i$;
\item[(b)]
$\calX=\{ X_i \}_{i=1}^n$;
\item[(c)]
set $R_j:=\End_{\sfD^\rmb(\mod A)}(X_j)$, then
\begin{align*}
\Hom_{\sfD^\rmb(\mod A)}(P_i,X_j) \cong \begin{cases} 
R_j \text{ as left $R_j$-modules}& (i=j) \\
0 & (i \ne j)
\end{cases}
\end{align*}
\end{itemize}
\end{itemize}
\end{Prop}

We also need the 2-term restrictions of the Koenig--Yang correspondences.

\begin{Prop}\label{Prop_twosilt_twosmc}\cite[Corollary 4.3]{BY}
The bijections $\silt A \to \tstr A$ and $\tstr A \to \smc A$ 
given in Proposition \ref{Prop_silt_smc} 
are restricted to bijections 
$\twosilt A \to \inttstr A$ and $\inttstr A \to \twosmc A$.
\end{Prop}

Thus, we shall use the following notation in the rest of this paper.

\begin{Def}\label{Def_setting_T_X}
Let $\calX \in \twosmc A$ correspond to $T \in \twosilt A$ in the bijections above.
We take families $(T_i)_{i=1}^n$ and $(X_i)_{i=1}^n$ satisfying
\begin{itemize}
\item[(a)]
$T=\bigoplus_{i=1}^n T_i$;
\item[(b)]
$\calX=\{ X_i \}_{i=1}^n$;
\item[(c)]
$(T_i)_{i=1}^n$ and $(X_i)_{i=1}^n$ give dual bases
of $K_0(\proj A)$ and $K_0(\mod A)$; more precisely,
\begin{align*}
\ang{T_i}{X_j} = \begin{cases}
\dim_K \End_{\sfD^\rmb(\mod A)}(X_j) & (i=j) \\
0 & (i \ne j)
\end{cases}.
\end{align*}
\end{itemize}
\end{Def}

\subsection{Cones of presilting objects}

Now, we define \textit{cones} $C(U), C^+(U) \subset K_0(\proj A)_\R$ 
for each object $U \in \sfK^\rmb(\proj A)$.
Decompose $U$ as $\bigoplus_{i=1}^m U_i \in \sfK^\rmb(\proj A)$ 
with $U_i$ indecomposable, and then set 
\begin{align*}
C(U)   &:= \{ a_1[U_1]+\cdots+a_m[U_m] \mid a_1,\ldots,a_m \in \R_{\ge 0} \}, \\
C^+(U) &:= \{ a_1[U_1]+\cdots+a_m[U_m] \mid a_1,\ldots,a_m \in \R_{> 0} \}.
\end{align*}
In particular, $C(0)=C^+(0)=\{0\}$ for $0 \in \sfK^\rmb(\proj A)$.

We mainly deal with $C(U)$ and $C^+(U)$ for $U \in \twopresilt A$.
In this case, $C^+(U)$ is a relative interior of the cone $C(U)$,
since the indecomposable direct summands of $U$ are linearly independent
in $K_0(\proj A)$ by Proposition \ref{Prop_presilt_silt}.
If $T \in \twosilt A$ and its indecomposable direct summands are $T_1,\ldots,T_n$,
then the cone $C(T)$ has exactly $n$ walls $C(T/T_i)$ with $i \in \{1,\ldots,n\}$,
and each wall $C(T/T_i)$ corresponds to the mutation of $T$ at $T_i$.

When we consider the intersection of cones for 2-term presilting objects, 
the following properties on uniqueness of presilting objects 
by \cite{DIJ} is crucial. 
We remark that (1) is an analogue of \cite[2.3, Theorem]{DK}.

\begin{Prop}\label{Prop_cone_summand}
Let $U, U'$ be (not necessarily basic) 2-term presilting objects 
in $\sfK^\rmb(\proj A)$.
Then, we have the following assertions.
\begin{itemize}
\item[(1)]
\cite[Theorem 6.5]{DIJ}
If $[U]=[U']$ in $K_0(\proj A)$, then $U \cong U'$ in $\sfK^\rmb(\proj A)$.
\item[(2)]
The following conditions are equivalent:
\begin{itemize}
\item[(a)]
$\add U=\add U'$,
\item[(b)] 
$C(U)=C(U')$,
\item[(c)]
$C^+(U)=C^+(U')$.
\end{itemize}
\item[(3)]
If $U'' \in \sfK^\rmb(\proj A)$ satisfies $\add U \cap \add U' = \add U''$,
then $C(U) \cap C(U') =C(U'')$.
\end{itemize}
\end{Prop}

\begin{proof}
Parts (2) and (3) immediately follow from (1) as in \cite[Corollary 6.7]{DIJ}.
\end{proof}

In particular, if $T' \in \twosilt A$ is not isomorphic to $T \in \twosilt A$, 
then $C^+(T) \cap C(T') = \emptyset$.

We here prepare some symbols.
For each $M \in \mod A$,
we define the following subcategories of $\mod A$:
\begin{itemize}
\item $M^\perp:=\{X \in \mod A \mid \Hom_A(M,X)=0\}$,
\item ${^\perp M} := \{X \in \mod A \mid \Hom_A(X,M)=0\}$,
\item $\Fac M:=
\{X \in \mod A \mid \text{there exists a surjection $M^{\oplus s} \to X$}\}$,
\item $\Sub M:=
\{X \in \mod A \mid \text{there exists an injection $X \to M^{\oplus s}$}\}$.
\end{itemize}
We write $\inj A$ for the category of finite-dimensional injective $A$-modules,
and let $\nu$ denote the Nakayama functor $\sfK^\rmb(\proj A) \to \sfK^\rmb(\inj A)$. 

Let $U \in \twopresilt A$.
First, $H^0(U)$ is $\tau$-rigid and $H^{-1}(\nu U)$ is 
$\tau^{-1}$-rigid by \cite[Lemma 3.4]{AIR}.
Thus, the torsion classes and torsion-free classes 
\begin{align*}
\ovcalT_U :={{}^\perp H^{-1}(\nu U)}, \quad
\calF_U := \Sub H^{-1}(\nu U), \quad
\calT_U :=\Fac H^0(U), \quad
\ovcalF_U := H^0(U)^\perp
\end{align*}
are functorially finite by \cite[Theorem 5.10]{AS}
(see \cite[Section 2]{MS} for the definition of functorially finite
subcategories in $\mod A$).
It is easy to see that $\calT_U \subset \ovcalT_U$ and 
that $\calF_U \subset \ovcalF_U$.

If $T \in \twosilt A$, then $H^0(T)$ is a support $\tau$-tilting module in $\mod A$
by \cite[Proposition 3.6]{AIR}, 
so we have $\ovcalT_T=\calT_T$ and $\ovcalF_T=\calF_T$ from \cite[Proposition 2.16]{AIR}.
Conversely, if $\ovcalT_U=\calT_U$ and $\ovcalF_U=\calF_U$ hold,
then $U$ is 2-term silting by \cite[Theorems 2.12, 3.2]{AIR}.

One of the main results of \cite{AIR} is that the set $\twosilt A$ has bijections to 
the set $\ftors A$ of functorially finite torsion classes
and the set $\ftorf A$ of functorially finite torsion-free classes in $\mod A$.

\begin{Prop}\label{Prop_silt_ftors}\cite[Theorems 2.7, 3.2]{AIR}
There exist bijections 
\begin{align*}
\twosilt A &\to \ftors A, & T &\mapsto \ovcalT_T=\calT_T;\\
\twosilt A &\to \ftorf A, & T &\mapsto \ovcalF_T=\calF_T.
\end{align*}
\end{Prop}

Therefore, for $U \in \twopresilt A$, there uniquely exists 
$T \in \twosilt A$ satisfying $U \in \add T$ and $\ovcalT_T=\ovcalT_U$.
We call this $T$ the \textit{Bongartz completion} of $U$.
Similarly, we can uniquely take 
$T' \in \twosilt A$ satisfying $U \in \add T$ and  
$\calT_{T'}=\calT_U$,
and such $T'$ is called the \textit{co-Bongartz completion} of $U$.

Now, we can state one of the main results of this section.

\begin{Thm}\label{Thm_presilt_TF}
Let $U \in \twopresilt A$.
Then, the subset $C^+(U) \subset K_0(\proj A)_\R$ is a TF equivalence class
with
\begin{align*}
C^+(U) &= \{\theta \in K_0(\proj A)_\R \mid
\ovcalT_\theta=\ovcalT_U, \ \ovcalF_\theta=\ovcalF_U\} \\
&= \{\theta \in K_0(\proj A)_\R \mid
\calT_\theta=\calT_U, \ \calF_\theta=\calF_U\}.
\end{align*}
\end{Thm}

%

We will prove this in the rest of this subsection.
First, the torsion classes $\calT_U \subset \ovcalT_U$ satisfy the following property.

\begin{Lem}\label{Lem_summand_silt}
\cite[Proposition 2.9]{AIR}
Let $U \in \twopresilt A$ and $T \in \twosilt A$.
Then, $U \in \add T$ if and only if
$\calT_U \subset \ovcalT_T \subset \ovcalT_U$.
\end{Lem}

%
%

By using this, we have the following properties.

\begin{Lem}\label{Lem_two_Bongartz_inj}
Let $U \in \twopresilt A$ and $T,T'$ be the Bongartz completion and 
the co-Bongartz completion of $U$, respectively.
\begin{itemize}
\item[(1)]
We have $\add T \cap \add T'=\add U$.
\item[(2)]
If $U' \in \twopresilt A$ satisfies $\ovcalT_{U'}=\ovcalT_U$ and $\calT_{U'}=\calT_U$,
then $U \cong U'$.
\item[(3)]
Let $V \in \twopresilt A$, then $V \in \add U$
if and only if $\calT_V \subset \calT_U \subset \ovcalT_U \subset \ovcalT_V$.
\end{itemize}
\end{Lem}

\begin{proof}
(1)
Clearly, $\add T \cap \add T' \supset \add U$,
so we let $V \in (\add T \cap \add T') \setminus \add U$ be indecomposable
and deduce a contradiction.
We consider the mutation $T''$ of $T$ at $V$,
then $U$ is a direct summand of $T''$, since $V \notin \add U$.
Thus, $\calT_U \subset \ovcalT_{T''} \subset \ovcalT_U$ 
by Lemma \ref{Lem_summand_silt}.
On the other hand, due to $V \in \add T \cap \add T'$ and the choices of $T$ and $T'$,
we have $\calT_V \subset \calT_{T'} = \calT_U$ and  
$\ovcalT_U = \ovcalT_T \subset \ovcalT_V$.
Therefore, $\calT_V \subset \ovcalT_{T''} \subset \ovcalT_V$,
so Lemma \ref{Lem_summand_silt} implies that $V$ is a direct summand of $T''$.
This contradicts that $T''$ is the mutation of $T$ at $V$.
Therefore, $\add T \cap \add T'=\add U$.

(2)
By assumption, $T$ and $T'$ are 
the Bongartz completion and the co-Bongartz completion also of $U'$.
Then, (1) implies $\add U=\add U'$.
Since $U$ and $U'$ are basic, we get $U \cong U'$.

(3)
The ``only if'' part is easy.
For the ``if'' part, assume $\calT_V \subset \calT_U \subset \ovcalT_U \subset \ovcalT_V$.
This implies $\calT_V \subset \calT_{T'} \subset \ovcalT_T \subset \ovcalT_V$,
so we obtain that $V \in \add T \cap \add T'$
from Lemma \ref{Lem_summand_silt}.
Then, (1) tells us that $V \in \add U$.
\end{proof}

In order to connect numerical torsion pairs and functorially finite torsion pairs,
the following result by Yurikusa \cite{Yurikusa} is important.

\begin{Prop}\label{Prop_Yurikusa}\cite[Proposition 3.3]{Yurikusa}
Let $U \in \twopresilt A$ and $\theta \in C^+(U)$.
Then, 
\begin{align*}
(\ovcalT_\theta,\calF_\theta)=(\ovcalT_U,\calF_U), \quad
(\calT_\theta,\ovcalF_\theta)=(\calT_U,\ovcalF_U), \quad
\calW_\theta = \ovcalT_U \cap \ovcalF_U.
\end{align*}
\end{Prop}

This implies that $C^+(U)$ is contained in a TF equivalence class.
Thus, to prove Theorem \ref{Thm_presilt_TF}, it remains to show the converse.
For this purpose, we recall the following result
on 2-term simple-minded collections by Br\"{u}stle--Yang \cite{BY}.

\begin{Lem}\label{Lem_modA_[1]}\cite[Remark 4.11]{BY}
Let $\calX \in \twosmc A$, then
every $X \in \calX$ satisfies $X \in \brick A$ or $X \in (\brick A)[1]$
up to isomorphisms in $\sfD^\rmb(\mod A)$.
\end{Lem}

Therefore, it is natural to consider the intersections
$\calX \cap \mod A$ and $\calX[-1] \cap \mod A$ for $\calX \in \twosmc A$.
We call a subset $\calS$ of $\brick A$ a \textit{semibrick} in $\mod A$
if $\Hom_A(S,S')=0$ holds for any two different (hence, non-isomorphic) 
elements $S,S' \in \calS$, and write $\sbrick A$ for the set of semibricks in $\mod A$.
By definition, 
the sets $\calX \cap \mod A$ and $\calX[-1] \cap \mod A$ are semibricks.

In \cite{Asai}, we introduced the notions of left-finiteness and right-finiteness 
of semibricks:
a semibrick $\calS$ is said to be \textit{left finite} (resp.~\textit{right finite}) if 
the smallest torsion (resp.~torsion-free) class $\sfT(\calS)$ (resp.~$\sfF(\calS)$) 
in $\mod A$ 
containing $\calS$ is functorially finite in $\mod A$.
We write $\fLsbrick A$ (resp.~$\fRsbrick A$) for the set of 
left finite (resp.~right finite) semibricks.
In that paper, we obtained the following bijections.

\begin{Prop}\label{Prop_smc_sbrick}\cite[Theorem 2.3]{Asai}
There exist bijections
\begin{align*}
\twosmc A &\to \fLsbrick A, & \calX &\mapsto \calX \cap \mod A, \\
\twosmc A &\to \fRsbrick A, & \calX &\mapsto \calX[-1] \cap \mod A.
\end{align*}
Moreover, if $\calX \in \twosmc A$ corresponds to $T \in \twosilt A$
in the bijections in Proposition \ref{Prop_twosilt_twosmc},
then $(\sfT(\calX \cap \mod A),\sfF(\calX[-1] \cap \mod A))
=(\ovcalT_T, \calF_T)=(\calT_T, \ovcalF_T)$.
\end{Prop}

Now, we can prove Theorem \ref{Thm_presilt_TF}.

\begin{proof}[Proof of Theorem \ref{Thm_presilt_TF}]
It suffices to prove that $\theta \in C^+(U)$ holds if and only if
$(\ovcalT_\theta,\calF_\theta)=(\ovcalT_U,\calF_U)$ and
$(\calT_\theta,\ovcalF_\theta)=(\calT_U,\ovcalF_U)$.

The ``only if'' part is nothing but Proposition \ref{Prop_Yurikusa}.

Thus, it remains to show the ``if'' part.
We assume $(\ovcalT_\theta,\calF_\theta)=(\ovcalT_U,\calF_U)$ and
$(\calT_\theta,\ovcalF_\theta)=(\calT_U,\ovcalF_U)$.
We take the Bongartz completion $T=\bigoplus_{i=1}^n T_i$ of $U$ with $T_i$ indecomposable.

If $\theta \in C(T)$, then there exists some direct summand $U'$ of $T$
such that $\theta \in C^+(U')$.
Then, $\ovcalT_{U'}=\ovcalT_\theta=\ovcalT_U$ and $\calT_{U'}=\calT_\theta=\calT_U$ 
follow from Proposition \ref{Prop_Yurikusa} and the assumption,
and we get $U' \cong U$ by Lemma \ref{Lem_two_Bongartz_inj} (2).
Thus, it is sufficient to prove $\theta \in C(T)$.

Since $[T_1],\ldots,[T_n]$ is a basis of $K_0(\proj A)_\R$,
there exist $a_1,\ldots,a_n \in \R$ such that $\theta = \sum_{i=1}^n a_i[T_i]$.
Thus, we prove $a_i \ge 0$ for all $i$.
We take $\calX=\{ X_i \}_{i=1}^n \in \twosmc A$ corresponding 
to $T \in \twosilt A$ as in Definition \ref{Def_setting_T_X}.
For each $i$, we have 
$\theta(X_i)=a_i \dim_K \End_{\sfD^\rmb(\mod A)}(X_i)$.

By Proposition \ref{Prop_smc_sbrick}, 
if $X_i \in \mod A$, then $X_i$ belongs to $\ovcalT_T=\ovcalT_U=\ovcalT_\theta$,
so $\theta(X_i) \ge 0$ holds;
and otherwise, $X_i$ belongs to $\mod A[1]$ by Lemma \ref{Lem_modA_[1]}, and
we get $X_i[-1] \in \calF_T=\calF_U=\calF_\theta$ and $\theta(X_i[-1])<0$,
which impiles $\theta(X_i)>0$.
Therefore, $a_i \ge 0$ holds for all $i$,
and we obtain $\theta \in C(T)$.
\end{proof}

\subsection{All chambers come from silting objects}

We conclude this section by results on the chambers
of the wall-chamber structure of $K_0(\proj A)_\R$.
Recall from Corollary \ref{Cor_TF_chamber} that there exists a bijection
\begin{align*}
\Chamber(A) \to \TF_n(A), \qquad C \mapsto 
\text{(the TF equivalence class containing $C$)}.
\end{align*}
In this subsection,
we would like to show that 
\begin{align*}
\Chamber(A) = \TF_n(A).
\end{align*}
This is nothing but to say that any chamber itself is a TF equivalence class.
To prove it, it is sufficient to check that
any TF equivalence class $[\theta]$ of dimension $n$ admits 
some $T \in \twosilt A$ such that $[\theta]=C^+(T)$,
since $C^+(T) \subset K_0(\proj A)_\R$ is clearly an open set.
For this purpose, we aim to get the following properties on cones and walls
(recall that $\Wall$ is the union $\bigcup_{M \in \mod A \setminus \{0\}} \Theta_M$).

\begin{Thm}\label{Thm_chamber_cone}
We have equations
\begin{align*}
\coprod_{T \in \twosilt A} C^+(T) &= K_0(\proj A) \setminus \overline{\Wall}, \\
\coprod_{T \in \twosilt A} C^+(T)_\Q &= K_0(\proj A)_\Q \setminus \Wall.
\end{align*}
In particular, there exists a bijection
\begin{align*}
\twosilt A \to \Chamber(A), \qquad T \mapsto C^+(T). 
\end{align*}
Therefore, all chambers are TF equivalence classes, so $\Chamber(A)=\TF_n(A)$.
\end{Thm}

We remark that the injectivity of (3) was proved by 
Br\"{u}stle--Smith--Treffinger \cite[Corollary 3.29]{BST}
in a different method.

To show this, 
we use the bijection entitled the \textit{Happel--Reiten--Smal{\o} tilt} \cite{HRS}
between the set of intermediate t-structures in $\sfD^\rmb(\mod A)$ and 
the set of torsion pairs in $\mod A$.
This bijection sends a torsion pair $(\calT,\calF)$ to 
the intermediate t-structure $(\calU,\calV)$ in $\sfD^\rmb(\mod A)$, where
\begin{align*}
\calU &:= \{X \in \sfD^\rmb(\mod A) \mid 
\text{$H^0(X)\in\calT$ and $H^i(X)=0$ for $i>0$}\}, \\
\calV &:= \{X \in \sfD^\rmb(\mod A) \mid 
\text{$H^0(X)\in\calF$ and $H^i(X)=0$ for $i<0$}\},
\end{align*}
and the inverse map sends an intermediate t-structure $(\calU,\calV)$ to 
the torsion pair $(\calU \cap \mod A, \calV \cap \mod A)$ in $\mod A$.

For our purpose,
it is important to know which torsion pairs in $\mod A$ correspond to 
intermediate t-structures with length heart. 
The answer is given by the following result of \cite{BY}.

\begin{Prop}\label{Prop_inttstr_ftors}\cite[Theorem 4.9]{BY}
There exist two bijections $\inttstr A \to \ftors A$ and $\inttstr A \to \ftorf A$
given by $(\calU, \calV) \mapsto \calU \cap \mod A$
and $(\calU, \calV) \mapsto \calV \cap \mod A$, respectively.
Moreover, for $T \in \twosilt A$,
we have $T[{<}0]^\perp \cap \mod A=\ovcalT_T$ and $T[{>}0]^\perp \cap \mod A=\calF_T$. 
\end{Prop}

We also need the following observation.

\begin{Prop}\label{Prop_Bridgeland}
Let $\theta \in K_0(\proj A)_\Q$.
If $\calW_\theta=\{0\}$, then
there exists $T \in \twosilt A$ such that $\theta \in C^+(T)$.
\end{Prop}

\begin{proof}
Clearly,
$(\ovcalT_\theta,\calF_\theta)=(\calT_\theta,\ovcalF_\theta)$.

By a similar argument to the proof of \cite[Lemma 7.1]{Bridgeland},
we show that
the heart $\calH$ of the corresponding t-structure 
$(\calU,\calV)$ for the torsion pair 
$(\ovcalT_\theta,\calF_\theta)=(\calT_\theta,\ovcalF_\theta)$ is 
a length category.
In this case, we can check that
every nonzero $X \in \calH$ admits a triangle 
$X' \to X \to X'' \to X[1]$ with $X' \in \calF_\theta[1]$ and 
$X'' \in \ovcalT_\theta=\calT_\theta$.
Thus, we get $\theta(X)=\theta(X')+\theta(X'') > 0$,
since $X'$ or $X''$ is nonzero.
By the assumption $\theta \in K_0(\proj A)_\Q$,
there exists some positive rational number $q \in \Q_{>0}$ such that 
any nonzero $X \in \calH$ satisfies $\theta(X) \ge q$.
Therefore, any $X \in \calH$ must have a composition series in the abelian category $\calH$.
This means that $\calH$ is a length category.

Thus, there exists some $T \in \twosilt A$
such that $(T[{<}0]^\perp, T[{>}0]^\perp)=(\calU,\calV)$
by Proposition \ref{Prop_twosilt_twosmc},
and then, $(\ovcalT_\theta,\calF_\theta)=(\calT_\theta,\ovcalF_\theta)=(\ovcalT_T,\calF_T)$
by Proposition \ref{Prop_inttstr_ftors}.
Applying Theorem \ref{Thm_presilt_TF}, we get $\theta \in C^+(T)$.
\end{proof}

Now, we can show Theorem \ref{Thm_chamber_cone} as follows.

\begin{proof}[Proof of Theorem \ref{Thm_chamber_cone}]
We remark that $\coprod_{T \in \twosilt A} C^+(T)$ 
and $\coprod_{T \in \twosilt A} C^+(T)_\Q$
are disjoint unions by Proposition \ref{Prop_cone_summand}.

We first show the second equation.
It suffices to show that 
there exists some $T \in \twosilt A$ such that $\theta \in C^+(T)$
if and only if $\calW_\theta=\{0\}$ in the case that $\theta \in K_0(\proj A)_\Q$.
The ``only if'' part follows 
from Propositions \ref{Prop_silt_ftors} and \ref{Prop_Yurikusa},
and the ``if'' part has been proved in Proposition \ref{Prop_Bridgeland}.

Next we show the first equation.
We set $\Cone:=\coprod_{T \in \twosilt A} C^+(T)$.

As in the proof of the ``only if'' part above, 
$\Cone \subset K_0(\proj A)_\R \setminus \Wall$ follows,
and since the left-hand side is an open subset of $K_0(\proj A)_\R$, we have
$\Cone \subset K_0(\proj A)_\R \setminus \overline{\Wall}$.

On the other hand, if 
$\theta \in K_0(\proj A)_\R \setminus \overline{\Wall}$,
then there exists a convex neighborhood 
$N \subset K_0(\proj A)_\R \setminus \overline{\Wall}$ of $\theta$,
and we can take $\theta' \in N \cap K_0(\proj A)_\Q$.
Then, there exists $T \in \twosilt A$ such that $\theta' \in C^+(T)$ 
by the second equation.
Because $N$ is convex, the line segment $[\theta,\theta']$ is contained in 
$N \subset K_0(\proj A)_\R \setminus \overline{\Wall}$,
thus $\theta$ and $\theta'$ are TF equivalent by Theorem \ref{Thm_W_constant}.
Since $C^+(T)$ is a TF equivalence class by Theorem \ref{Thm_presilt_TF},
$\theta$ also belongs to $C^+(T)$.
Thus, $\theta \in \Cone$ as desired.
We have $K_0(\proj A)_\R \setminus \overline{\Wall} \subset \Cone$.
We get the first equation.

Clearly, $\Cone=\coprod_{T \in \twosilt A} C^+(T)$
is the decomposition into the connected components.
Therefore, the map 
$\twosilt A \to \Chamber(A)$, $T \mapsto C^+(T)$ is well-defined 
and bijective.
Since $C^+(T)$ for $T \in \twosilt A$ is a TF equivalence class
by Theorem \ref{Thm_presilt_TF}, all chambers are TF equivalence classes.
Thus, Corollary \ref{Cor_TF_chamber} gives the remaining statement.
\end{proof}

\begin{Rem}
If $\theta \in K_0(\proj A)$, then
there exist semibricks $\calS,\calS'$ such that
$\ovcalT_\theta=\sfT(\calS)$ and $\ovcalF_\theta=\sfF(\calS')$
by Theorem \ref{Thm_widely_gen},
since the set of simple objects of a wide subcategory is a semibrick.
If $\theta$ additionally satisfies $\calW_\theta=\{0\}$,
then we can show that $\calS \cup \calS'[1] \in \twosmc A$.
Our first proof of Theorem \ref{Thm_chamber_cone} was based on this fact.
\end{Rem}

\section{Reduction of the wall-chamber structures}

\subsection{$\tau$-tilting reduction and the local wall-chamber structures}

Recall that we obtained 
an injection from $\twopresilt A$ to the set of TF equivalence classes
sending $U$ to $C^+(U)$ in Theorem \ref{Thm_presilt_TF}.
For $\theta \in C^+(U)$, the $\theta$-semistable subcategory $\calW_\theta$ is 
a wide subcategory $\calW_U:=\ovcalT_U \cap \ovcalF_U$.
The wide subcategory $\calW_U$ was investigated by \cite{Jasso, DIRRT}
as $\tau$-tilting reduction in the context of $\tau$-rigid pairs, 
and they found that $\calW_U$ is equivalent to the module category $\mod B$
for an algebra $B$ constructed from $U$;
see \textup{\cite[Theorem 3.8]{Jasso} and \cite[Theorem 4.12]{DIRRT}}.

The corresponding result for 2-term presilting objects is given as follows,
and we write a direct proof by using \cite{Asai,IY}
for the convenience of the readers.
We can check that
our $\phi$ is compatible with their original equivalence
$\Hom_A(H^0(T),?) \colon \calW_U \to \mod B$ as in the proof of \cite[Theorem 3.16]{Asai}.

\begin{Prop}\label{Prop_Jasso_wide}
Let $U \in \twopresilt A$.
Define $T \in \twosilt A$ as its Bongartz completion,
and set an algebra $B$ as $\End_{\sfK^\rmb(\proj A)}(T)/\langle e \rangle$,
where $e$ is the idempotent $(T \to U \to T) \in \End_{\sfK^\rmb(\proj A)}(T)$.
Then, we have an equivalence
\begin{align*}
\phi:=\Hom_A(T,?) \colon \calW_U \to \mod B.
\end{align*}
Moreover, let 
$T=\bigoplus_{i=1}^n T_i$ and $U=\bigoplus_{i=m+1}^n T_i$ with
$T_i$ indecomposable, take $\calX \in \twosmc A$ corresponding to $T$,
and define $X_1,X_2,\ldots,X_n \in \calX$ as in Definition \ref{Def_setting_T_X}.
Then, $\{X_1,X_2,\ldots,X_m\}$ is the set of simple objects in $\calW_U$,
and sent to the set of simple $B$-modules.
\end{Prop}

\begin{proof}
First, we check that $X_i \in \calW_U$ for $i \in \{1,\ldots,m\}$.
Let $T'$ be the mutation of $T$ at $T_i$, then since $U$ is a direct summand of $T/T_i$,
we have $\ovcalT_{T'} \subset \ovcalT_{T/T_i} \subset \ovcalT_U = \ovcalT_T$
from Lemma \ref{Lem_two_Bongartz_inj}.
Then, from the proof of \cite[Theorem 3.12]{Asai},
$X_i \in \mod A$ follows,
and moreover, $X_i \in \ovcalT_T \cap \ovcalF_{T'}$.
By Lemma \ref{Lem_two_Bongartz_inj} again, 
we get $X_i \in \ovcalT_T \cap \ovcalF_{T'} \subset \ovcalT_U \cap \ovcalF_U = \calW_U$.

Now, we recall the equivalence
$\Hom_{\sfD^\rmb(\mod A)}(T,?)\colon T[{\ne}0]^\perp 
\to \mod \End_{\sfK^\rmb(\proj A)}(T)$
by Iyama--Yang \cite[Proposition 4.8]{IY},
which sends $\calX$ to the set of simple $\End_{\sfK^\rmb(\proj A)}(T)$-modules.
Then, the same strategy as the proof of \cite[Theorem 3.15]{Asai} yields that
the equivalence above is restricted to an equivalence
$\phi=\Hom_A(T,?) \colon \calW_U \to \mod B$ between their
Serre subcategories,
sending $\{X_1,X_2,\ldots,X_m\}$ to the simple objects in $\mod B$.
\end{proof}

Moreover, Jasso proved that the set $\twosilt B$ has a bijection from the subset 
$\twosilt_U A:=\{V \in \twosilt A \mid U \in \add V\}$ of $\twosilt A$
compatible with Proposition \ref{Prop_silt_ftors}.
Actually, this bijection can be extended to 
2-term presilting objects.

\begin{Prop}\label{Prop_Jasso_silt}
Let $U \in \twopresilt A$, and consider the functor
\begin{align*}
\red:=\Hom_{\sfK^\rmb(\proj A)}(T,?)/[U] \colon 
\sfK^\rmb(\proj A) \to \sfK^\rmb(\proj B), 
\end{align*}
where $[U]$ is the ideal in $\sfK^\rmb(\proj A)$ 
consisting of the maps factoring through
objects in $\add U$.
\begin{itemize}
\item[(1)]
\cite[Theorems 3.14, 3.16, 4.12]{Jasso}
The functor $\red$ induces a bijection 
$\red \colon \twosilt_U A \to \twosilt B$ satisfying
$\ovcalT_{\red(V)}=\phi(\ovcalT_V \cap \calW_U) \in \ftors B$
and $\ovcalF_{\red(V)}=\phi(\ovcalF_V \cap \calW_U) \in \ftorf B$.
\item[(2)]
Set $\twopresilt_U A:=\{V \in \twopresilt A \mid U \in \add V\}$.
Then, the functor $\red$ gives a bijection $\red \colon \twopresilt_U A \to \twopresilt B$.
\end{itemize}
\end{Prop}

\begin{proof}
The part (2) is verified by considering 
the Bongartz completion and the co-Bongartz completion of each $V$
and using Lemma \ref{Lem_two_Bongartz_inj}.
\end{proof}

Next, we will investigate the relationship between the Grothendieck groups
$K_0(\proj A)_\R$ and $K_0(\proj B)_\R$ in terms of the functor $\phi$.
For this purpose, we define a subset $N_U \subset K_0(\proj A)_\R$ by
\begin{align*}
N_U:=\{ \theta \in K_0(\proj A)_\R \mid 
\calT_U \subset \calT_\theta \subset \ovcalT_\theta \subset \ovcalT_U \}
\end{align*}
for each $U \in \twopresilt A$.
If $\theta \in N_U$, then $\ovcalF_\theta \subset \ovcalF_U$;
hence, $\calW_\theta \subset \calW_U$.
By definition, $N_U$ is a union of some TF equivalent classes in $K_0(\mod A)_\R$.

The following property is easy to deduce, but crucial.

\begin{Lem}\label{Lem_N_U_open}
Let $U \in \twopresilt A$.
Then, $N_U \subset K_0(\proj A)_\R$ is an open neighborhood of $C^+(U)$.
\end{Lem}

\begin{proof}
We can check that 
$\calT_U \subset \calT_\theta \subset \ovcalT_\theta \subset \ovcalT_U$
if and only if both $H^0(U) \in \calT_\theta$ and $H^{-1}(\nu U) \in \calF_\theta$ hold.
The latter conditions can be written as 
a collection of finitely many strict linear inequalities on $\theta$,
so $N_U$ is an open subset of $K_0(\proj A)_\R$.
Moreover, Proposition \ref{Prop_Yurikusa} tells us that $C^+(U)$ is contained in $N_U$.
\end{proof}

In the setting of Proposition \ref{Prop_Jasso_wide},
$S^B_i:=\phi(X_i)$ is a simple $B$-module, 
so set $P^B_i \in \proj B$ as the projective cover of $S^B_i$.
We define a linear map $\pi \colon K_0(\proj A)_\R \to K_0(\proj B)_\R$ by 
\begin{align*}
\pi(\theta):=\sum_{i=1}^m \frac{\theta(X_i)}{d_i}[P^B_i],
\end{align*}
where $d_i:=\dim_K \End_A(X_i)=\dim_K \End_B(\phi(X_i))$.
Then, $\pi$ satisfies the following nice properties in the subset $N_U$.

\begin{Lem}\label{Lem_p(N_U)}
Let $U \in \twopresilt A$.
Then, the following assertions hold.
\begin{itemize}
\item[(1)]
The restriction $\pi|_{N_U} \colon N_U \to K_0(\proj B)_\R$ is surjective.
\item[(2)]
For any $\theta \in K_0(\proj A)_\R$ and $M \in \calW_U$, 
we have $\pi(\theta)(\phi(M))=\theta(M)$.
\item[(3)]
For any $\theta \in N_U$, we have the following equations in $\mod B$:
\begin{align*}
\phi(\ovcalT_\theta \cap \calW_U) &= \ovcalT_{\pi(\theta)}, &
\phi(\calF_\theta \cap \calW_U) &= \calF_{\pi(\theta)}, \\
\phi(\calT_\theta \cap \calW_U) &= \calT_{\pi(\theta)}, &
\phi(\ovcalF_\theta \cap \calW_U) &= \ovcalF_{\pi(\theta)}, &
\phi(\calW_\theta) &= \calW_{\pi(\theta)}.
\end{align*}
\item[(4)]
Let $\theta, \theta' \in N_U$.
The elements $\theta$ and $\theta'$ are TF equivalent in $K_0(\proj A)_\R$ if and only if 
$\pi(\theta)$ and $\pi(\theta')$ are TF equivalent in $K_0(\proj B)_\R$.
In particular, if $\pi(\theta)=\pi(\theta')$, 
then $\theta$ and $\theta'$ are TF equivalent.
\end{itemize}
\end{Lem}

\begin{proof}
(1)
Since $N_U$ is an open neighborhood of $[U]$,
the image $\pi(N_U)$ is an open neighborhood of $0$ in $K_0(\proj B)_\R$.
Moreover, $N_U=\R_{>0} \cdot N_U$ follows from the definition,
so $\pi(N_U)=\R_{>0} \cdot \pi(N_U)$.
Thus, $\pi(N_U)$ must be $K_0(\proj B)_\R$.

(2)
Since $\theta$ and $\pi(\theta)$ are linear maps,
we may assume that $M$ is simple in $\calW_U$; in other words, 
$M=X_j$ for some $j \in \{1,2,\ldots,m\}$.
Then, $\phi(M)=S_j^B$.
Because $P_i^B$ is the projective cover of $S_i^B$ in $\mod B$ 
for $i \in \{1,2,\ldots,m\}$,
we have $\ang{P_j^B}{S_j^B}=d_j$ and $\ang{P_i^B}{S_j^B}=0$ for $i \ne j$.
Therefore, $\pi(\theta)(\phi(M))=\pi(\theta)(S_j^B)=\theta(X_j)$ as desired.

(3)
From (2), we have 
$\phi(\ovcalT_\theta \cap \calW_U) \subset \ovcalT_{\pi(\theta)}$ and 
$\phi(\calF_\theta \cap \calW_U) \subset \calF_{\pi(\theta)}$.
Thus, if 
$(\phi(\ovcalT_\theta \cap \calW_U),\phi(\calF_\theta \cap \calW_U))$ 
is a torsion pair in $\mod B$,
then we get
$\phi(\ovcalT_\theta \cap \calW_U) = \ovcalT_{\pi(\theta)}$ and 
$\phi(\calF_\theta \cap \calW_U) = \calF_{\pi(\theta)}$.

To show that $(\phi(\ovcalT_\theta \cap \calW_U),\phi(\calF_\theta \cap \calW_U))$ 
is a torsion pair in $\mod B$,
it suffices to check that $(\ovcalT_\theta \cap \calW_U, \calF_\theta \cap \calW_U)$
is a torsion pair in $\calW_U$.
Clearly, $\Hom_A(\ovcalT_\theta \cap \calW_U, \calF_\theta \cap \calW_U)=0$.
Moreover, for all $M \in \calW_U$, 
there exists a short exact sequence $0 \to M' \to M \to M'' \to 0$ 
with $M' \in \ovcalT_\theta$ and $M'' \in \calF_\theta$ in $\mod A$.
Since $\theta \in N_U$, we get $M' \in \ovcalT_\theta \subset \ovcalT_U$.
On the other hand, since $M \in \calW_U$, we get $M' \in \ovcalF_U$.
Thus, $M'$ also belongs to $\calW_U$, and so does $M''$.
Therefore, $0 \to M' \to M \to M'' \to 0$ is a short exact sequence
with $M' \in \ovcalT_\theta \cap \calW_U$ and 
$M'' \in \ovcalF_\theta \cap \calW_U$ in $\calW_U$.
We have proved that $(\ovcalT_\theta \cap \calW_U, \calF_\theta \cap \calW_U)$
is a torsion pair in $\calW_U$.

From the argument in the first paragraph, 
$\phi(\ovcalT_\theta \cap \calW_U) = \ovcalT_{\pi(\theta)}$ and 
$\phi(\calF_\theta \cap \calW_U) = \calF_{\pi(\theta)}$ hold.

Similarly, we obtain
$\phi(\calT_\theta \cap \calW_U) = \calT_{\pi(\theta)}$ and 
$\phi(\ovcalF_\theta \cap \calW_U) = \ovcalF_{\pi(\theta)}$.

Finally, since $\calW_\theta \subset \calW_U$ by $\theta \in N_U$,
\begin{align*}
\phi(\calW_\theta)
=\phi(\ovcalT_\theta \cap \ovcalF_\theta \cap \calW_U)
=\phi(\ovcalT_\theta \cap \calW_U) \cap \phi(\ovcalF_\theta \cap \calW_U)
=\ovcalT_{\pi(\theta)} \cap \ovcalF_{\pi(\theta)}
=\calW_{\pi(\theta)}.
\end{align*}

(4)
\cite[Theorem 3.12]{Jasso} tells us the following property: 
let $\calT,\calT'$ be two torsion classes in $\mod A$ satisfying 
$\calT_U \subset \calT,\calT' \subset \ovcalT_U$,
then $\calT \cap \calW_U=\calT' \cap \calW_U$ holds if and only if $\calT=\calT'$.
This fact and (3) imply the assertion.
\end{proof}

Therefore, the wall-chamber structure of $N_U \subset K_0(\proj A)_\R$ 
recovers that of $K_0(\proj B)_\R$ via $\pi$ as follows.

\begin{Thm}\label{Thm_N_U_TF_union}
Let $U \in \twopresilt A$. Then, we have the following properties.
\begin{itemize}
\item[(1)]
For any $\theta \in N_U$ and $M \in \calW_U$,
the wall $\Theta_{\phi(M)}$ coincides with $\pi(\Theta_M \cap N_U)$.
\item[(2)]
The linear map $\pi$ induces a bijection 
\begin{align*}
\{\textup{TF equivalence classes in $N_U$}\} & \to 
\{\textup{TF equivalence classes in $K_0(\proj B)_\R$}\}, \\
[\theta] &\mapsto \pi([\theta]).
\end{align*}
\item[(3)]
We have the following commutative diagram:
\begin{align*}
\begin{xy}
(  0,  8) *+{\twopresilt_U A}="1",
( 80,  8) *+{\twopresilt B}="2",
(  0, -8) *+{\{\textup{TF equivalence classes in $N_U$}\}}="3",
( 80, -8) *+{\{\textup{TF equivalence classes in $K_0(\proj B)_\R$}\}}="4".
\ar_{C^+} "1";"3"
\ar_{C^+} "2";"4"
\ar^{\red}_{\cong} "1";"2"
\ar^{\pi}_{\cong} "3";"4"
\end{xy}.
\end{align*}
\end{itemize}
\end{Thm}

\begin{proof}
(1)
This follows from Lemma \ref{Lem_p(N_U)} (3).

(2)
Note that $\pi|_{N_U} \colon N_U \to K_0(\proj B)_\R$ is surjective 
by Lemma \ref{Lem_p(N_U)} (1).
Then, Lemma \ref{Lem_p(N_U)} (4) yields that 
the linear map $\pi$ sends each TF equivalence class $[\theta]$ in $N_U$
to a TF equivalence class $\pi([\theta])$ in $K_0(\proj B)_\R$, 
and that this correspondence is bijective.

(3)
Let $V \in \twopresilt_U A$.
By Theorem \ref{Thm_presilt_TF}, $C^+(V)$ is a TF equivalence class in $K_0(\proj A)_\R$,
and it is contained in $N_U$ by Lemma \ref{Lem_two_Bongartz_inj}.
Then, the bijection in (2) sends $C^+(V)$ to 
the TF equivalence class $\pi(C^+(V))$ in $K_0(\proj B)_\R$,
which must be $C^+(\red(V))$ by Proposition \ref{Prop_Jasso_silt} and Lemma
\ref{Lem_p(N_U)} (3).
\end{proof}

We remark that $N_U$ itself is not very important 
to investigate the wall-chamber structure of $K_0(\proj B)_\R$.
By using Theorem \ref{Thm_N_U_TF_union},
we can obtain all information on walls, chambers and TF equivalence classes 
in $K_0(\proj B)_\R$
from a subset $N' \subset N_U$ and the restriction $\pi|_{N'} \colon N' \to \pi(N')$
as long as $0 \in K_0(\proj B)_\R$ is in the interior of $\pi(N')$ in $K_0(\proj B)_\R$.  

We conclude this subsection by giving an example.

\begin{Ex}
Let $A$ be the path algebra $K(1 \to 2 \to 3)$, and
take injections $f \colon P_2 \to P_1$ and $g \colon P_3 \to P_1$.
We can check that $U:=(P_2 \xrightarrow{f} P_1)$ 
and $V:=(P_3 \xrightarrow{g} P_1)$ belong to $\twopresilt A$.

Consider the $\tau$-tilting reduction at $U$.
We here use the setting of Proposition \ref{Prop_Jasso_wide}
for the Bongartz completion $T=T_1 \oplus T_2 \oplus T_3$ of $U$ 
as $T_1=P_1$, $T_2=P_3$, and $T_3=U$.

Then, the algebra $B$ is isomorphic to $K(1 \to 2)$,
and the simple objects of $\calW_U$ are $X_1=\Coker g=P_1/P_3$ and $X_2=S_3$.
The indecomposable objects of $\calW_U$ are $X_1$, $X_2$ and $P_1$,
which correspond to $S^B_1$, $S^B_2$ and $P^B_1$ in $\mod B$, respectively.
The walls for the indecomposable modules in $\calW_U$ are
\begin{align*}
&\Theta_{X_1}=\R_{\ge 0}([P_1]-[P_2]) \oplus \R[P_3], \quad
\Theta_{X_2}=\R[P_1] \oplus \R[P_2], \\
&\Theta_{P_1}=\R_{\ge 0}([P_1]-[P_2]) \oplus \R_{\ge 0}([P_2]-[P_3]).
\end{align*}
Since $H^0(U)=S_1$ and $H^{-1}(\nu U)=S_2$, 
we get $N_U=\R_{>0}[P_1] \oplus \R_{>0}(-[P_2]) \oplus \R[P_3]$,
so 
\begin{align*}
&\Theta_{X_1} \cap N_U=\R_{>0}([P_1]-[P_2]) \oplus \R[P_3], \quad
\Theta_{X_2} \cap N_U=\R_{>0}[P_1] \oplus \R_{>0}(-[P_2]), \\
&\Theta_{P_1} \cap N_U=\R_{>0}([P_1]-[P_2]) \oplus \R_{\ge 0}([P_1]-[P_3]).
\end{align*}

Since $[X_1]=[S_1]+[S_2]$ and $[X_2]=[S_3]$ in $K_0(\mod A)$,
the linear map $\pi$ sends $[P_1],[P_2],[P_3]$ to $[P^B_1],[P^B_1],[P^B_2]$, respectively.
Thus, 
\begin{align*}
\pi(\Theta_{X_1} \cap N_U)&=\R[P^B_2], &
\pi(\Theta_{X_2} \cap N_U)&=\R[P^B_1], &
\pi(\Theta_{P_1} \cap N_U)&=\R_{\ge 0}([P^B_1]-[P^B_2]).
\end{align*}

On the other hand, the set $\twosilt_U A$ has exactly five elements:
\begin{align*}
T^{(1)}&:=P_1 \oplus P_3 \oplus U, & 
T^{(2)}&:=P_1 \oplus V \oplus U, &
T^{(3)}&:=V \oplus P_3[1] \oplus U, \\
T^{(4)}&:=P_2[1] \oplus P_3 \oplus U, &
T^{(5)}&:=P_2[1] \oplus P_3[1] \oplus U, 
\end{align*}
and the functor $\red=\Hom_{\sfK^\rmb(\proj A)}(T,?)/[U]$ acts to 
their indecomposable direct summands as 
\begin{align*}
P_1 & \mapsto P^B_1, & P_3 & \mapsto P^B_2, & V& \mapsto (P^B_2 \to P^B_1), &
P_2[1] & \mapsto P^B_1[1], & P_3[1] & \mapsto P^B_2[1].
\end{align*}
The corresponding elements in $K_0(\proj A)_\R$ are sent by $\pi$ as 
\begin{align*}
[P_1] & \mapsto [P^B_1], & [P_3] & \mapsto [P^B_2], & [V] & \mapsto [P^B_1]-[P^B_2], &
-[P_2] & \mapsto -[P^B_1], & -[P_3] & \mapsto -[P^B_2].
\end{align*}
Thus, $\pi$ is compatible with the bijection $\red \colon \twosilt_U A \to \twosilt B$.

Therefore, the local wall-chamber structures around $[U] \in K_0(\proj A)_\R$ and 
$0 \in K_0(\proj B)_\R$ are depicted as in Figure \ref{figure_local} below.
\begin{figure}
\includegraphics[width=9.5cm]{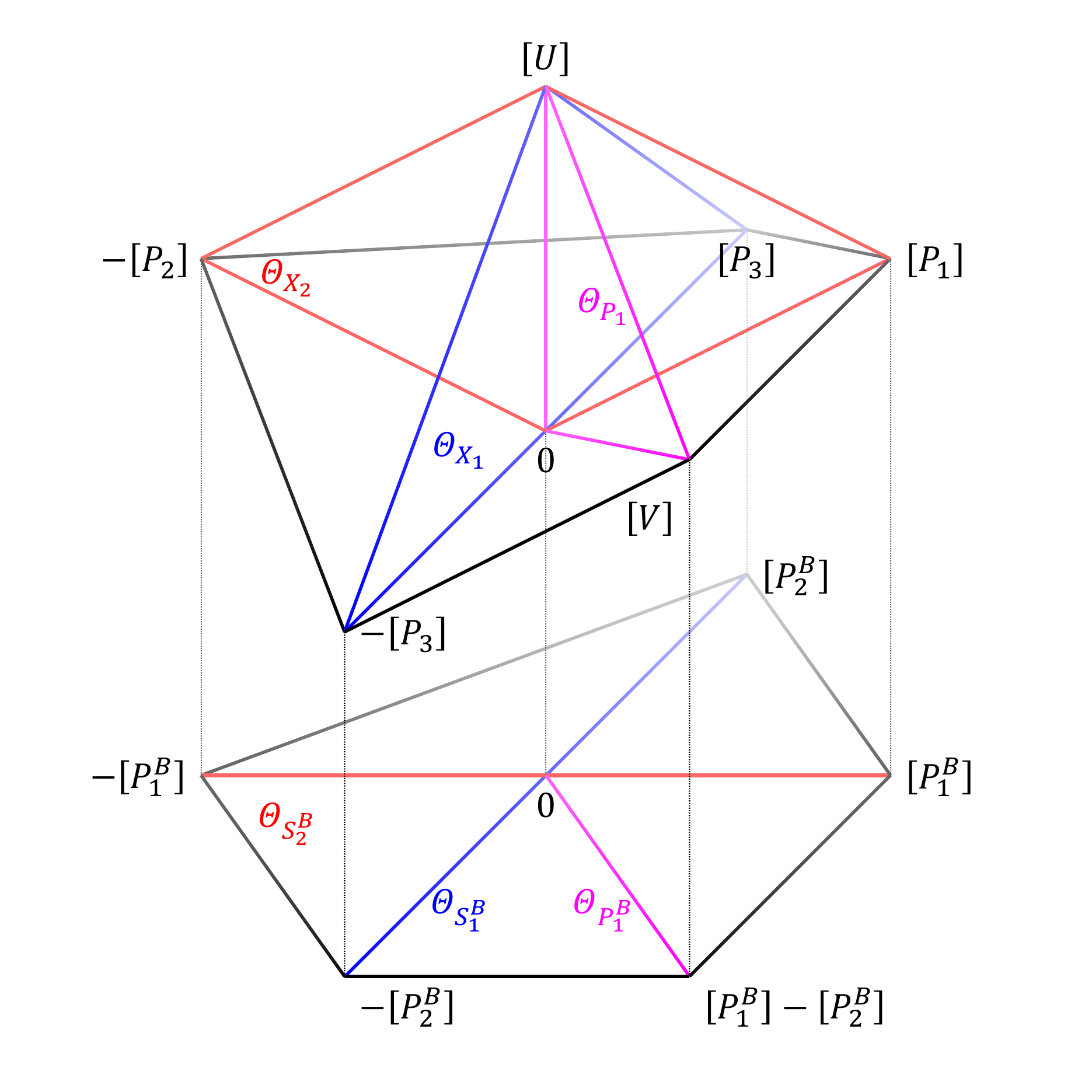}
\caption{The local wall-chamber structures around $[U] \in K_0(\proj A)_\R$ and 
$0 \in K_0(\proj B)_\R$}\label{figure_local}
\end{figure}
\end{Ex}

\subsection{Application to $\tau$-tilting finiteness}

By \cite[Theorem 3.8]{DIJ} and Proposition \ref{Prop_silt_ftors}, 
an algebra $A$ is said to be \textit{$\tau$-tilting finite} 
if $\twosilt A$ is a finite set,
or equivalently, if the set $\twoipresilt A$ of 2-term indecomposable presilting objects 
in $\sfK^\rmb(\proj A)$ is a finite set.
In this section, we give a proof of the following characterization of 
$\tau$-tilting finiteness in terms of the cones for 2-term silting objects
by using the subset $N_U$ for each $U \in \twopresilt A$.

\begin{Thm}\label{Thm_finite_cover_eq}
The algebra $A$ is $\tau$-tilting finite 
if and only if $K_0(\proj A)_\R=\bigcup_{T \in \twosilt A} C(T)$.
\end{Thm}

The ``if'' part follows from \cite[Theorem 5.4, Corollary 6.7]{DIJ}.
We give a direct proof for the convenience of the readers.

\begin{Prop}\label{Prop_finite_cover}
If $A$ is $\tau$-tilting finite, 
then $K_0(\proj A)_\R=\bigcup_{T \in \twosilt A} C(T)$.
\end{Prop}

\begin{proof}
Set subsets $F_1,F_2 \subset K_0(\proj A)_\R$ by
\begin{align*}
F_1:=\bigcup_{T \in \twosilt A} C(T) \quad \text{and} \quad
F_2:=\bigcup_{
\begin{smallmatrix}
U \in \twopresilt A \\ |U|=n-2
\end{smallmatrix}} C(U).
\end{align*}
Since $F_2 \subset F_1$,
it suffices to show $F_1 \setminus F_2=K_0(\proj A)_\R \setminus F_2$.

For $U \in \twopresilt A$ with $|U|=n-1$,
we can take the two distinct elements $T,T'$ in $\twopresilt_U A$
by Proposition \ref{Prop_presilt_silt} (3),
and then, 
$C'(U):=C^+(T) \cup C^+(U) \cup C^+(T')$ is open in $K_0(\proj A)_\R \setminus F_2$.
Thus,
\begin{align*}
F_1 \setminus F_2=\bigcup_{\begin{smallmatrix}
U \in \twopresilt A \\ |U|=n-1,n
\end{smallmatrix}} C^+(U)=
\bigcup_{\begin{smallmatrix}
U \in \twopresilt A \\ |U|=n-1
\end{smallmatrix}} C'(U)
\end{align*}
is an open subset of $K_0(\proj A)_\R \setminus F_2$.
On the other hand, $F_1 \subset K_0(\proj A)_\R$ is closed since $\#(\twosilt A)<\infty$,
so $F_1 \setminus F_2$ is a closed subset of $K_0(\proj A)_\R \setminus F_2$.
Clearly, $F_1 \setminus F_2$ is nonempty. 

These three statements imply that $F_1 \setminus F_2 =K_0(\proj A)_\R \setminus F_2$,
since $K_0(\proj A)_\R \setminus F_2$ is connected.
Now, the assertion follows.
\end{proof}

We next prove the ``if'' part.
We remark that this was conjectured by Demonet \cite[Question 3.48]{Demonet}, 
and that Zimmermann--Zvonareva \cite{ZZ} have given a different proof.

\begin{Prop}\label{Prop_cover_finite}
If $K_0(\proj A)_\R=\bigcup_{T \in \twosilt A} C(T)$,
then $A$ is $\tau$-tilting finite.
\end{Prop}

\begin{proof}
Clearly, any $\theta \in C(T) \setminus \{0\}$ admits 
some nonzero $V \in \twopresilt A$ such that
$\theta \in C^+(V)$.
Lemma \ref{Lem_two_Bongartz_inj} and Proposition \ref{Prop_Yurikusa} yield that 
if $U \in \twoipresilt A$ and $V \in \twopresilt A$ satisfy $U \in \add V$,
then $C^+(V) \subset N_U$.
Thus, the assumption $K_0(\proj A)_\R=\bigcup_{T \in \twosilt A} C(T)$ implies that
$K_0(\proj A)_\R \setminus \{0\}=\bigcup_{U \in \twoipresilt A} N_U$.

We consider the canonical sphere $\Sigma \subset K_0(\proj A)_\R$;
more precisely,
\begin{align*}
\Sigma:=\left\{ 
\sum_{i=1}^n a_i[P_i] \in K_0(\proj A)_\R \mid \sum_{i=1}^n a_i^2=1 \right\}.
\end{align*}
Then, $\Sigma=\bigcup_{U \in \twoipresilt A} (N_U \cap \Sigma)$.
Since $N_U \cap \Sigma$ is an open subset of 
the compact space $\Sigma$ for every $U \in \twoipresilt A$,
there exists a finite set $I \subset \twoipresilt A$ such that 
$\Sigma=\bigcup_{U \in I} (N_U \cap \Sigma)$.
Clearly, this implies that 
$K_0(\proj A)_\R \setminus \{0\}=\bigcup_{U \in I} N_U$.

It is sufficient to show that $\twoipresilt A=I$,
so let $V \in \twoipresilt A$.
Since $[V] \in K_0(\proj A)_\R \setminus \{0\}$, 
there exists some $U \in \twoipresilt A$ such that $[V] \in N_U$.
By Lemma \ref{Lem_two_Bongartz_inj} and Proposition \ref{Prop_Yurikusa}, 
we have $V \in \add U$,
and since $U$ and $V$ are indecomposable, we get $V \cong U$.
Thus, $V \in I$.
Now, we get that $\twoipresilt A$ coincides with the finite set $I$,
and this means that $A$ is $\tau$-tilting finite.
\end{proof}

Now, we can show Theorem \ref{Thm_finite_cover_eq}.

\begin{proof}[Proof of Theorem \ref{Thm_finite_cover_eq}]
It follows from
Propositions \ref{Prop_finite_cover} and \ref{Prop_cover_finite}.
\end{proof}

\section{The wall-chamber structures for path algebras}\label{Sec_path}

In this section, we give a combinatorial algorithm to obtain 
the wall-chamber structure of $K_0(\proj A)_\R$ 
in the case that $A$ is a path algebra.
Throughout this section, $K$ is an algebraically closed field,
$Q$ is an acyclic quiver with $\#Q_0=n$, and $A:=KQ$.
We use the symbol $X^{Q_0}$ for the set of maps from $Q_0$ to $X$.

We need some fundamental facts on module varieties.
Let $\vecd \in (\Z_{\ge 0})^{Q_0}$ be a dimension vector,
and write $d_i=\vecd(i)$ for each $i \in Q_0$.
Then, we set 
\begin{align*}
\mod (A, \vecd):=\prod_{(\alpha \colon i \to j) \in Q_1} 
\Mat_K(d_j,d_i).
\end{align*}
This is exactly the set of representations of the quiver $Q$
whose dimension vectors are $\vecd$,
and we can regard $\mod (A, \vecd)$
as the set of all $A$-modules $M$ with $\dimv M=\vecd$.
By considering the Zariski topology,
$\mod (A, \vecd)$ has a structure of an algebraic variety,
so we call $\mod (A, \vecd)$ the \textit{module variety}
of $A$ associated to the dimension vector $\vecd$.
The module variety $\mod (A, \vecd)$ is clearly irreducible.
In particular, any nonempty open subset is dense in $\mod (A, \vecd)$.
We use the following property,
where $\vecc \le \vecd$ means that $c_i \le d_i$ for all $i \in Q_0$.

\begin{Prop}\label{Prop_Schofield}\cite[Lemma 3.1]{Schofield}
Let $\vecc \le \vecd \in (\Z_{\ge 0})^{Q_0}$ be 
two dimension vectors.
We define $F_{\vecc} \subset \mod (A,\vecd)$
consisting of all $M \in \mod (A,\vecd)$
admitting a submodule $L \subset M$ with $\dimv L=\vecc$.
Then, $F_{\vecc}$ is closed in $\mod (A,\vecd)$.
\end{Prop}

As in Theorem \ref{Thm_chamber_cone}, 
the union of the walls is more important than each wall itself, 
so we here define
\begin{align*}
\Theta_{\vecd}:=\bigcup_{M \in \mod (A,\vecd)} \Theta_M
\end{align*}
for every nonzero dimension vector $\vecd \in (\Z_{\ge 0})^{Q_0}$.
Actually, $\Theta_{\vecd}$ is realized 
as the wall of some module in $\mod (A,\vecd)$.

\begin{Lem}\label{Lem_max_wall}
Let $\vecd \in (\Z_{\ge 0})^{Q_0}$ be a nonzero dimension vector.
Then, there exists $M \in \mod (A, \vecd)$
such that $\Theta_M=\Theta_{\vecd}$;
hence $\Theta_{\vecd}$ is a rational polyhedral cone in $K_0(\proj A)_\R$.
\end{Lem}

\begin{proof}
We define $G_{\vecc}$ as the complement of
$F_{\vecc} \subset \mod (A,\vecd)$
for each dimension vector $\vecc \le \vecd$,
and set 
\begin{align*}
G:=\bigcap_{\vecc \le \vecd, \ G_{\vecc} \ne \emptyset} G_{\vecc}.
\end{align*}
By Proposition \ref{Prop_Schofield}, $G$ is the intersection of 
finitely many open dense subsets, so $G$ is also open and dense.
In particular, $G$ is nonempty.

We take $M \in G$.
Then, we have $\Theta_M \supset \Theta_{M'}$
for all $M' \in \mod (A, \vecd)$ by definition,
so $\Theta_M=\Theta_{\vecd}$ follows.
Since $\Theta_M$ is a rational polyhedral cone, 
so is $\Theta_{\vecd}$.
\end{proof}

A dimension vector $\vecd$ is called a \textit{Schur root}
if there exists an open dense subset $G$ of $\mod(A,\vecd)$
such that every $M \in G$ is indecomposable.

Since $A$ is hereditary, the homomorphism $K_0(\proj A) \to K_0(\mod A)$
sending $[P_i] \in K_0(\proj A)$ to $[P_i] \in K_0(\mod A)$ is isomorphic,
so from now on, we regard the Euler form $\langle ?,! \rangle$ as
a bilinear form $K_0(\mod A) \times K_0(\mod A) \to \Z$.

In this setting, a Schur root $\vecd$ is said to be 
\begin{itemize}
\item 
\textit{real} if $\ang{\vecd}{\vecd} = 1$,
\item 
\textit{imaginary} if $\ang{\vecd}{\vecd} \le 0$,
\item 
\textit{isotropic} if $\ang{\vecd}{\vecd} = 0$.
\end{itemize}

We next show that we can determine whether $\vecd$ is a Schur root
from the dimension of the wall $\Theta_d$ and 
the value $\ang{\vecd}{\vecd}$ of the Euler form.

\begin{Prop}\label{Prop_Schur_criterion}
Let $\vecd \in (\Z_{\ge 0})^{Q_0}$ be a nonzero dimension vector.
\begin{itemize}
\item[(1)]
Assume $\ang{\vecd}{\vecd} \ge 0$.
\begin{itemize}
\item[(a)]
The dimension vector $\vecd$ is a Schur root of $Q$
if and only if $\vecd$ is indivisible and $\dim \Theta_{\vecd}=n-1$.
\item[(b)]
There exist an integer $k \in \Z_{\ge 1}$ and 
a Schur root $\vecd'$ of $Q$ such that $\vecd=k\vecd'$ 
if and only if $\dim \Theta_{\vecd}=n-1$.
\end{itemize}
\item[(2)]
Assume $\ang{\vecd}{\vecd} < 0$.
Then, $\vecd$ is a Schur root of $Q$ if and only if 
$\dim \Theta_{\vecd}=n-1$.
\end{itemize}
\end{Prop}

\begin{proof}
(1)(a)
We first let $\vecd$ be a Schur root with 
$\ang{\vecd}{\vecd} \ge 0$.
In this case, $\vecd$ must be indivisible by \cite[Theorem 3.8]{Schofield}. 
Also, we can construct $\theta \in K_0(\proj A)$ such that 
there exists an open subset $G \subset \mod(A,\vecd)$ satisfying that
every $M \in G$ is $\theta$-stable from \cite[Theorem 6.1]{Schofield}.
This clearly implies that $\dim \Theta_{\vecd}=n-1$.

On the other hand, 
assume that $\vecd$ is indivisible and 
that $\dim \Theta_{\vecd}=n-1$.
Take the canonical decomposition $\vecd=\bigoplus_{i=1}^m \vecc_i$
(see \cite{Kac1}).
By Proposition \ref{Prop_Schofield}, this canonical decomposition implies that
every $M \in \mod(A,\vecd)$ has submodules $L_1,L_2$
such that $\dimv L_1=\vecc_i$ and $\dimv L_2=\vecd-\vecc_i$
for all $i$,
so any $\theta \in \Theta_{\vecd}$ must satisfy
$\theta(\vecc_i)=\theta(\vecd-\vecc_i)=0$.
Therefore, $\vecc_i \in \Q \vecd$ holds for every $i$,
since $\dim \Theta_{\vecd}=n-1$.
Since $\vecd$ is indivisible, we have $m=1$ and $\vecc_1=\vecd$.
Thus, $\vecd$ itself is a canonical decomposition,
so $\vecd$ is a Schur root.

(b)
Let $\vecd=k\vecd'$ with $k \in \Z_{\ge 1}$ and $\vecd'$ be a Schur root of $Q$.
Part (a) implies that $\dim \Theta_{\vecd'}=n-1$,
so in particular, there exists $M' \in \mod(A,\vecd')$ such that 
$\dim \Theta_{M'}=n-1$ by Lemma \ref{Lem_max_wall}.
Set $M:=(M')^{\oplus d}$, then $\Theta_M=\Theta_{M'}$, and $\dim \Theta_{M}=n-1$.
Therefore, $\dim \Theta_{\vecd}=n-1$.

Conversely, assume that $\dim \Theta_{\vecd} = n-1$.
As in the proof of (a),
we have $\vecc_i \in \Q \vecd$ for every $i$
in the canonical decomposition $\vecd=\bigoplus_{i=1}^m \vecc_i$,
and since $\ang{\vecd}{\vecd} \ge 0$,
every $\vecc_i$ must be indivisible by \cite[Theorem 3.8]{Schofield}.
Thus, $\vecd':=\vecc_1=\vecc_2=\cdots=\vecc_m$ is a Schur root and $\vecd=m\vecd'$.

(2)
The ``only if'' part follows from the same argument as (1)(a).
For the ``if'' part, 
we can show that $\vecc_i \in \Q \vecd$ holds for every $i$ in
the canonical decomposition $\vecd=\bigoplus_{i=1}^m \vecc_i$
in a similar way to (1)(a).
Since $\ang{\vecd}{\vecd} < 0$,
the root $\vecc_i$ is imaginary and non-isotropic for all $i$.
Therefore, \cite[Theorem 3.8]{Schofield} tells us that $m=1$.
\end{proof}

Now, we explicitly give the walls in the case that $\#Q_0=2$.

\begin{Ex}\label{Ex_Kronecker}
Assume that $Q$ is the $m$-Kronecker quiver:
\begin{align*}
\begin{xy}
( 0, 0)*+{1}="1",
(20, 0)*+{2}="2",
(10, 0)*+{\vdots},
\ar@<4mm> "1";"2"
\ar@<2mm> "1";"2"
\ar@<-4mm> "1";"2"
\end{xy}
\quad \text{($m$ arrows)}.
\end{align*}
By using \cite{Kac1}, we know all Schur roots. 
Then, for each Schur root $\vecd$, 
\cite[Theorem 6.1]{Schofield} guarantees the existence of 
an open dense subset $G \subset \mod(A,\vecd)$ and 
a stability condition $\theta$ such that every module in $G$ 
is $\theta$-stable.
Therefore, the wall-chamber structures of $K_0(\proj A)_\R$
for $m=0,1,2,3$ are given as follows (see also \cite[Figures 1--3]{Bridgeland}):
\begin{align*}
\includegraphics[width=3.8cm]{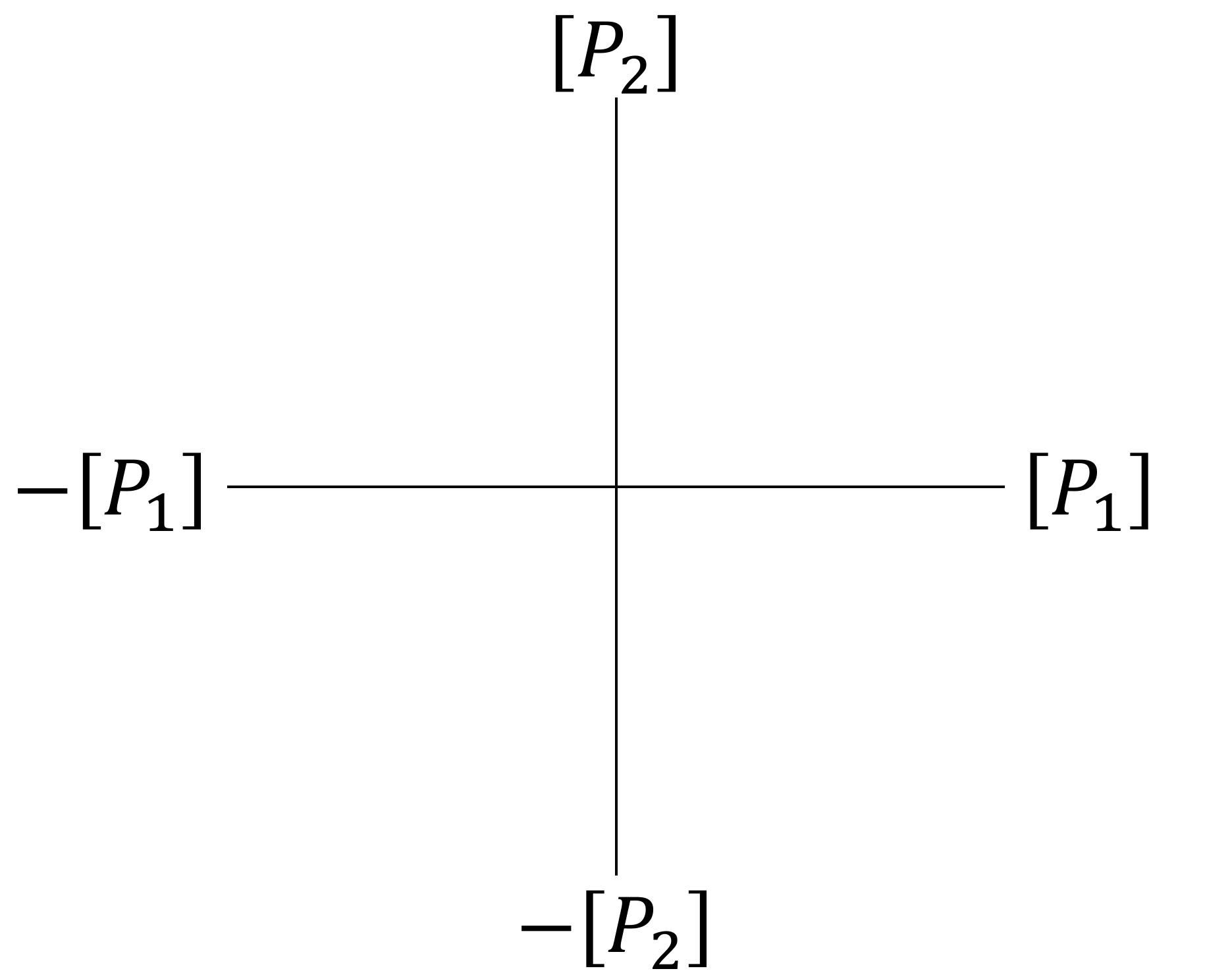} \ 
\includegraphics[width=3.8cm]{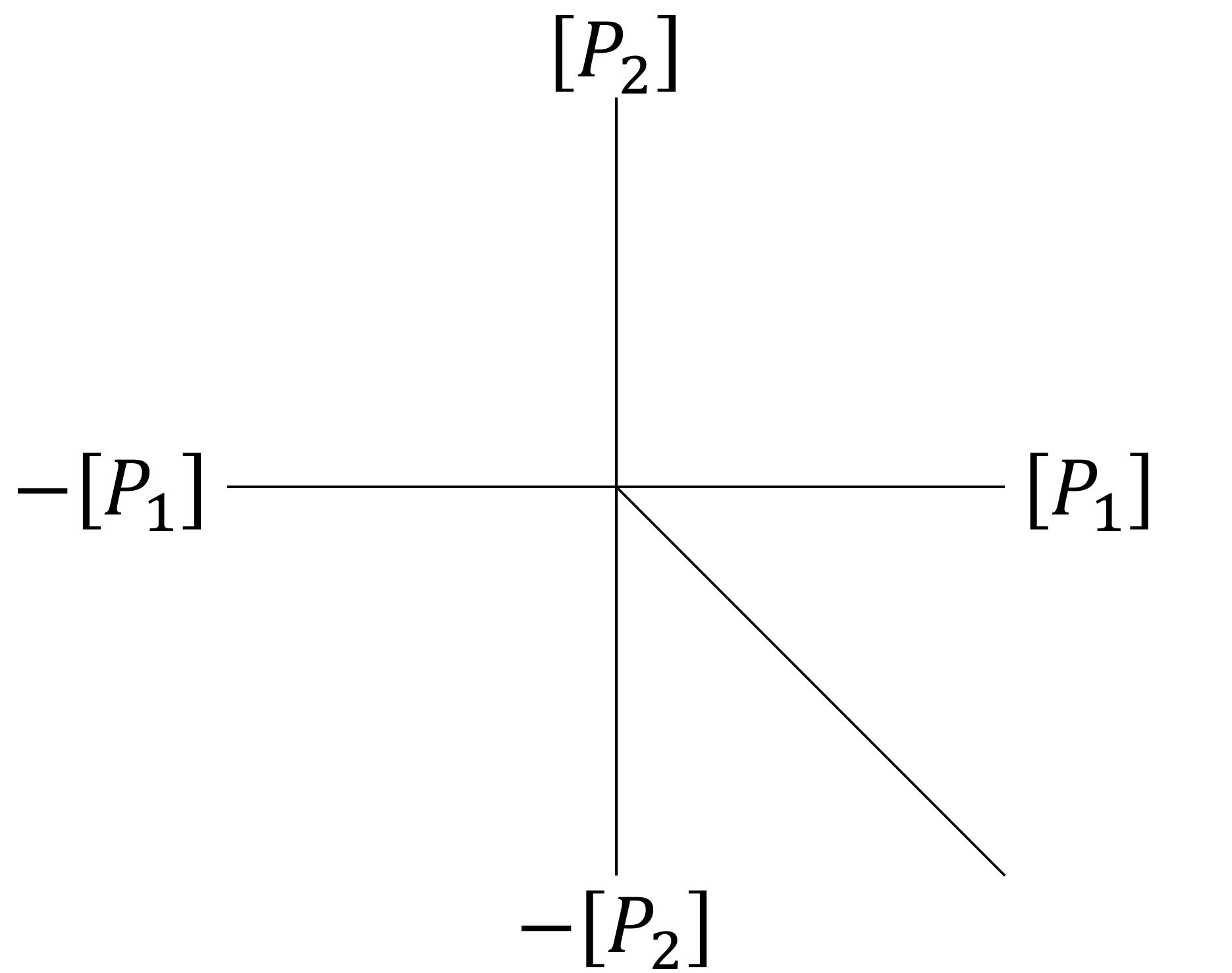} \ 
\includegraphics[width=3.8cm]{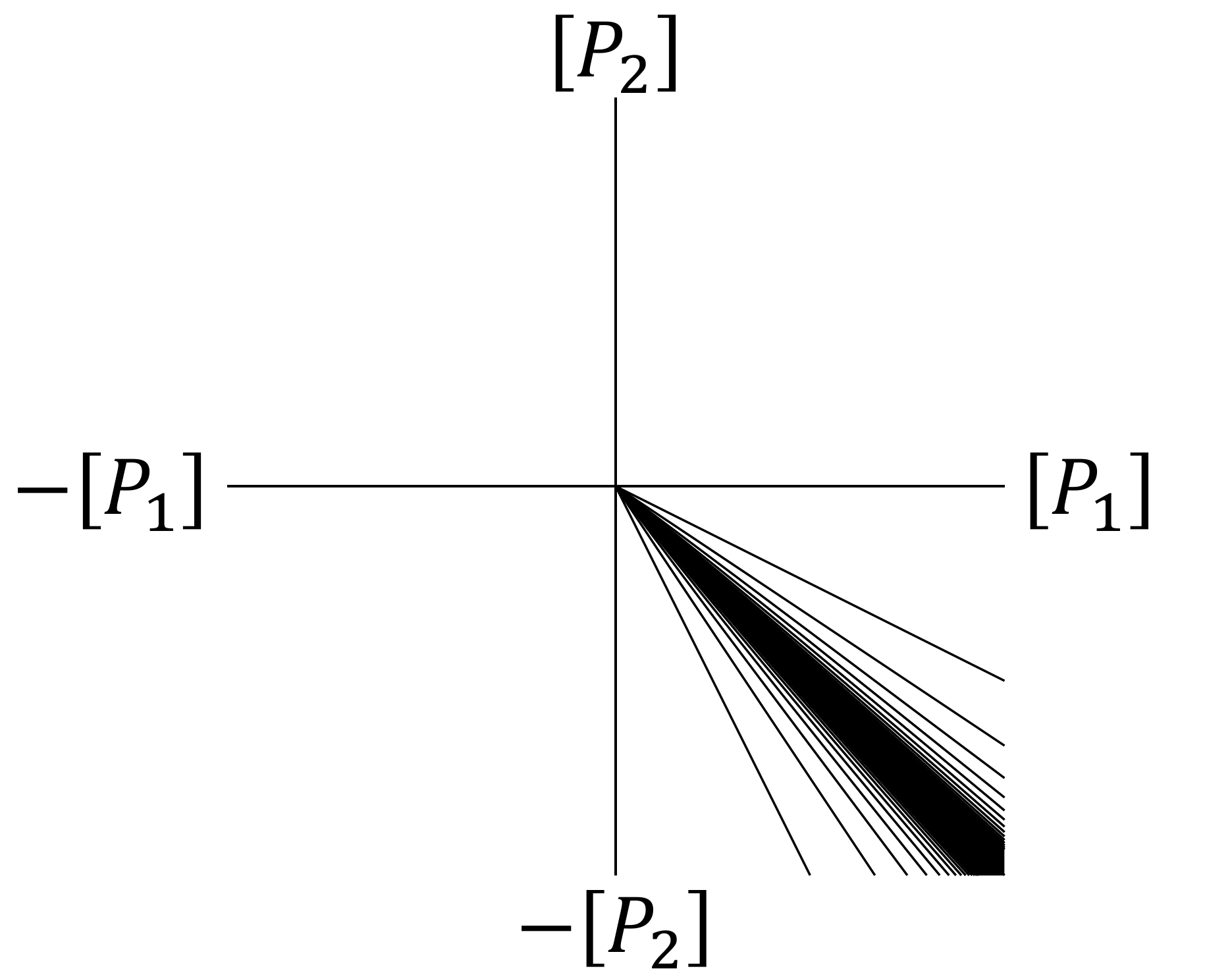} \ 
\includegraphics[width=3.8cm]{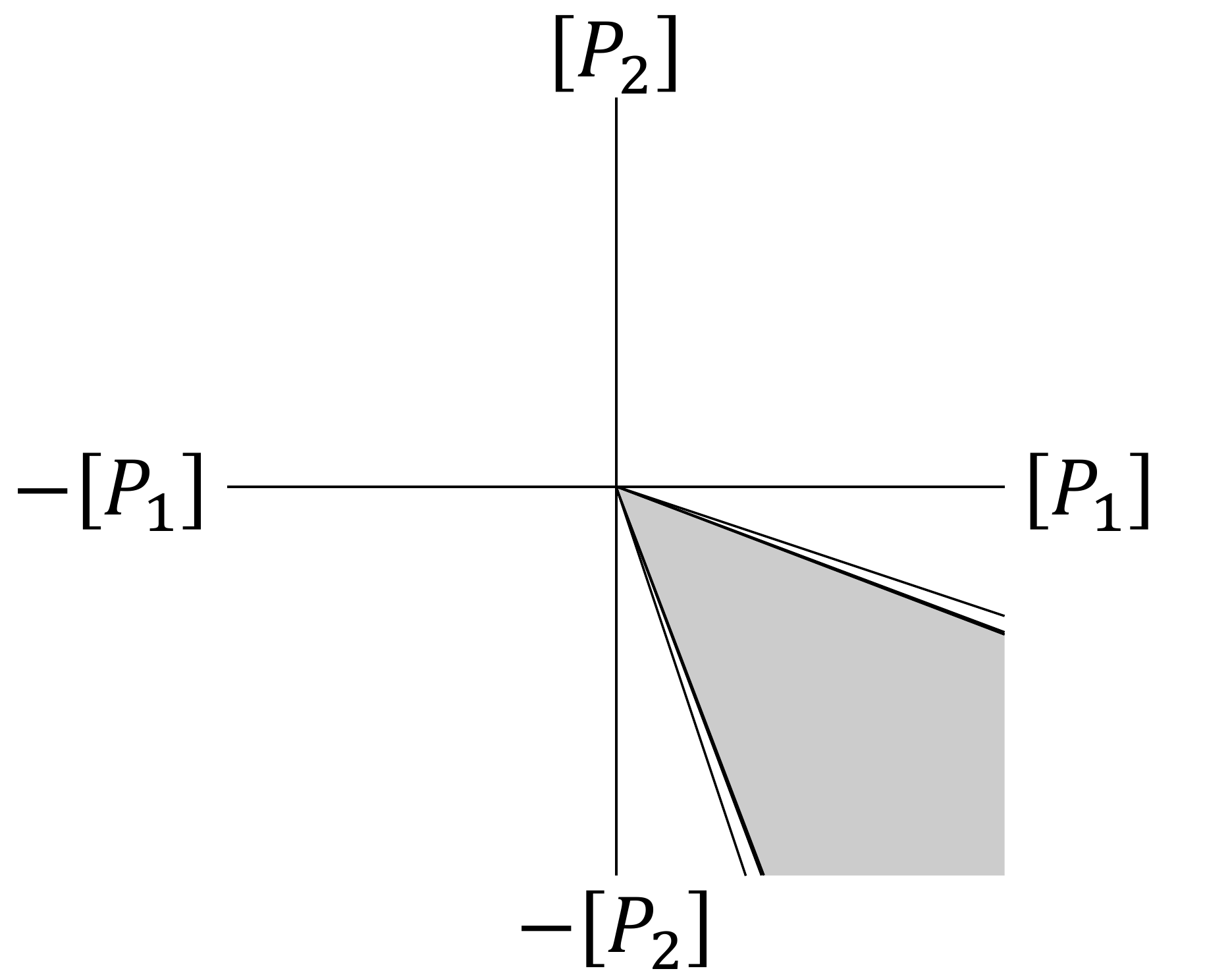}, 
\end{align*}
where there exists a wall $\R_{\ge 0}\cdot \theta$
for each rational point $\theta$ in the gray domain in the picture for $m=3$.
We write the wall $\Theta_{\vecd}$ for each 
$\vecd=(a,b) \in (\Z_{\ge 0})^2 \setminus \{0\}$ 
more explicitly below.

First, if $m=0$, then the real roots of $Q$ are $(0,1)$ and $(1,0)$,
and $Q$ admits no imaginary roots; hence,
\begin{align*}
\Theta_{\vecd}=\begin{cases}
\R[P_1] & (a=0) \\
\R[P_2] & (b=0) \\
\{0\} & (\text{otherwise})
\end{cases}.
\end{align*}

Second, consider the case that $m=1$.
In this case, the real roots of $Q$ are $(0,1)$, $(1,1)$ and $(1,0)$,
and no imaginary roots of $Q$ exist; hence, 
\begin{align*}
\Theta_{\vecd}=\begin{cases}
\R[P_1] & (a=0) \\
\R_{\ge 0}([P_1]-[P_2]) & (a=b) \\
\R[P_2] & (b=0) \\
\{0\} & (\text{otherwise})
\end{cases}.
\end{align*}

Next, assume that $m=2$.
Then, the real roots of $Q$ are $(i,i+1)$ and $(i+1,i)$ 
for all $i \in \Z_{\ge 0}$.
The unique imaginary root of $Q$ is $(1,1)$.
Thus, 
\begin{align*}
\Theta_{\vecd}=\begin{cases}
\R[P_1] & (a=0) \\
\R_{\ge 0}((i+1)[P_1]-i[P_2]) & ((a,b) \in \Z_{\ge 1} (i,i+1), \ i \in \Z_{\ge 1}) \\
\R_{\ge 0}([P_1]-[P_2]) & (a=b) \\
\R_{\ge 0}(i[P_1]-(i+1)[P_2]) & ((a,b) \in \Z_{\ge 1} (i+1,i), \ i \in \Z_{\ge 1}) \\
\R[P_2] & (b=0) \\
\{0\} & (\text{otherwise})
\end{cases}.
\end{align*}

We finally consider the case that $m \ge 3$.
In this case, the real roots of $Q$ are $(s_i,s_{i+1})$ and $(s_{i+1},s_i)$ 
for all $i \in \Z_{\ge 0}$, where the sequence $(s_i)_{i=0}^\infty$ is defined by
$s_0:=0$, $s_1:=1$, and $s_{i+2}:=ms_{i+1}-s_i$.
The imaginary roots of $Q$ are all $(a,b)$ satisfying $a^2+b^2-mab<0$.
Thus, 
\begin{align*}
\Theta_{\vecd}=\begin{cases}
\R[P_1] & (a=0) \\
\R_{\ge 0}(s_{i+1}[P_1]-s_i[P_2]) & ((a,b) \in \Z_{\ge 1}(s_i,s_{i+1}), \ i\in\Z_{\ge 1}) \\
\R_{\ge 0}(b[P_1]-a[P_2]) & (a^2+b^2-mab<0) \\
\R_{\ge 0}(s_i[P_1]-s_{i+1}[P_2]) & ((a,b) \in \Z_{\ge 1} (s_{i+1},s_i), \ i\in\Z_{\ge 1}) \\
\R[P_2] & (b=0) \\
\{0\} & (\text{otherwise})
\end{cases}.
\end{align*}
\end{Ex}

We set $\supp \vecd:=\#\{i \in Q_0 \mid d_i > 0\}$
for each dimension vector $\vecd \in (\Z_{\ge 0})^{Q_0}$.
Generalizing the example above by applying Lemma \ref{Lem_strongly_convex}, 
we can determine $\Theta_{\vecd}$ in the case
$1 \le \# \supp \vecd \le 2$.

\begin{Lem}\label{Lem_wall_init}
Suppose $\vecd \in (\Z_{\ge 0})^{Q_0}$ is a dimension vector
with $\#\supp \vecd \in \{1,2\}$.
\begin{itemize}
\item[(1)]
If $\# \supp \vecd = 1$ and $k \in \supp {\vecd}$,
then $\Theta_{\vecd}=\bigoplus_{i \ne k} \R[P_i]$.
\item[(2)]
Assume that $\# \supp \vecd = 2$
and that the full subquiver of $Q$ corresponding to $\supp \vecd \subset Q$ is 
\begin{align*}
\begin{xy}
( 0, 0)*+{k}="1",
(20, 0)*+{l}="2",
(10, 0)*+{\vdots},
\ar@<4mm> "1";"2"
\ar@<2mm> "1";"2"
\ar@<-4mm> "1";"2"
\end{xy}
\quad \textup{($m$ arrows)}
\end{align*} 
with $k,l \in \supp \vecd$ and $m \in \Z_{\ge 0}$.
We define $a,b \in \Z_{\ge 0}$ by $a:=d_k/\gcd(d_k,d_l)$ and $b:=d_l/\gcd(d_k,d_l)$.
Then, 
\begin{align*}
\Theta_{\vecd}=\begin{cases}
(\bigoplus_{i \ne k,l} \R[P_i]) \oplus \R_{\ge 0}(b[P_k]-a[P_l]) 
& (a^2+b^2-mab \le 1) \\
(\bigoplus_{i \ne k,l} \R[P_i])
& (\textup{otherwise})
\end{cases}.
\end{align*}
\end{itemize}
\end{Lem}

On the other hand, if $\# \supp \vecd \ge 3$, then we can apply Proposition \ref{Prop_wall_ext} to
obtain a recurrence relation on $\Theta_{\vecd}$.

\begin{Prop}\label{Prop_wall_recurrence}
Suppose $\vecd \in (\Z_{\ge 0})^{Q_0}$ is a dimension vector
with $\# \supp \vecd \ge 3$.
Then, $\Theta_{\vecd}$ is the smallest polyhedral cone of 
$K_0(\proj A)_\R$ containing
\begin{align*}
\bigcup_{0 < \vecc < \vecd} 
(\Theta_{\vecc} \cap \Theta_{\vecd-\vecc}).
\end{align*}
\end{Prop}

\begin{proof}
By Lemma \ref{Lem_max_wall}, $\Theta_{\vecd}$ itself is a polyhedral cone.
Proposition \ref{Prop_wall_ext} tells us that $\Theta_{\vecd}$ is 
the polyhedral cone of $K_0(\proj A)_\R$ generated by
\begin{align*}
\bigcup_{\begin{smallmatrix}
M',M'' \in \mod A \setminus \{0\}\\
(M'*M'') \cap \mod (A, \vecd) \ne \emptyset
\end{smallmatrix}} 
(\Theta_{M'} \cap \Theta_{M''})=
\bigcup_{0 < \vecc < \vecd}
\bigcup_{\begin{smallmatrix}
M' \in \mod(A,\vecc)\\
M'' \in \mod(A,\vecd-\vecc)
\end{smallmatrix}} 
(\Theta_{M'} \cap \Theta_{M''})
=\bigcup_{0 < \vecc < \vecd}
(\Theta_{\vecc} \cap \Theta_{\vecd-\vecc}),
\end{align*}
where we use Lemma \ref{Lem_max_wall} again for the latter equality.
\end{proof}

As a consequence, we can determine the wall-chamber structure of $K_0(\proj A)_\R$
for any path algebra $A$.

\begin{Thm}\label{Thm_wall_path}
Let $\vecd \in (\Z_{\ge 0})^{Q_0}$ be a nonzero dimension vector.
Then, $\Theta_{\vecd}$ is determined by Lemma \ref{Lem_wall_init} and 
Proposition \ref{Prop_wall_recurrence}.
\end{Thm}

We end this paper by giving an example of Theorem \ref{Thm_wall_path}.

\begin{Ex}\label{Ex_pptx1}
Let $Q$ be a quiver $1 \rightrightarrows 2 \to 3$.
The following picture is the wall-chamber structure of $K_0(\proj A)_\R$ on
the subset 
\begin{align*}
\{ a_1[P_1]-a_2[P_2]-a_3[P_3] \mid a_1,a_2, a_3 \ge 0, \ a_1+a_2+a_3=1 \}.
\end{align*}

\begin{align*}
\includegraphics[width=5.7cm]{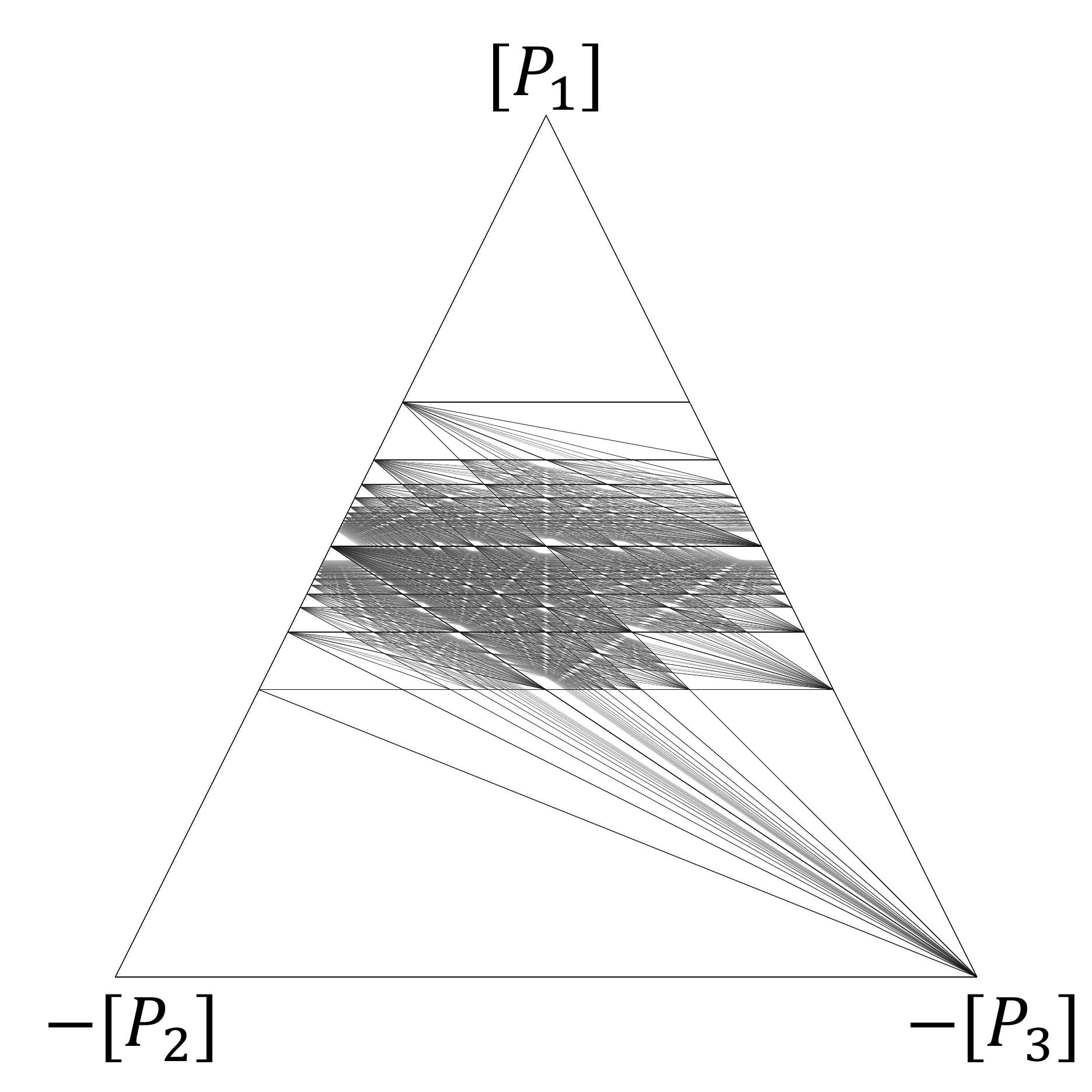}
\end{align*}

Compare our figure in $K_0(\proj A)_\R$ with the diagram in $K_0(\mod A)_\R$ in \cite[Example 11.3.9]{DW},
then we find that the chambers in our figure are sent to the triangles 
expressing tilting modules in their diagram
under the linear map $f \colon K_0(\proj A) \to K_0(\mod A)$
sending $[P_i] \in K_0(\proj A)$ to $[P_i] \in K_0(\mod A)$.
\end{Ex}

\section*{Funding}

This work was supported by Japan Society for the Promotion of Science KAKENHI JP16J02249
and JP19K14500.

\section*{Acknowledgement}

The author thanks to Osamu Iyama for thorough instruction.
For Theorem \ref{Thm_chamber_cone}, 
he appreciates Aaron Chan improving his proof.
Proposition \ref{Prop_cover_finite} was originally a conjecture of Laurent Demonet, 
and has already been proved by Alexander Zimmermann and Alexandra Zvonareva,
so the author is grateful to them for admitting him to include it in this paper.
The author is also thankful to Gustavo Jasso and Jan Schr\"{o}er 
for fruitful discussion in his visits of the University of Bonn.

\end{document}